\documentclass[11pt]{article}
\usepackage[toc,page]{appendix}
\usepackage{hyperref}
\usepackage[T1]{fontenc}
\usepackage{amsfonts}
\usepackage[english]{babel}
\usepackage{a4wide,times}

\usepackage[latin1]{inputenc}
\usepackage{amssymb}
\usepackage{graphicx}
\usepackage{epsfig}
\usepackage{amsbsy}
\usepackage{verbatim}
\usepackage{color}
\usepackage{mathrsfs}
\usepackage{makeidx}
\usepackage{amsmath}
\usepackage{amsthm}

\theoremstyle{plain}
\newtheorem{thm}{Theorem}[section]

\newtheorem{lem}[thm]{Lemma}
\newtheorem{prop}[thm]{Proposition}

\theoremstyle{definition}

\usepackage{dsfont}

\theoremstyle{remark}
\newtheorem{rem}{\bf Remark}[section]
\theoremstyle{remark}

\newtheorem{com*}{\bf Comment}

\usepackage{fancyhdr}

\makeatletter
\def \newequation#1#2{
 \@definecounter{#1}
 \@namedef{the#1}{\hbox{#2}}
 \@namedef{#1}{$$\refstepcounter{#1}}
 \@namedef{end#1}{
    \eqno \csname the#1\endcsname $$\global\@ignoretrue
    }
}
\newcommand{\E}{\mathds E}
\newcommand{\PP}{\mathds P}
\newcommand{\R}{\mathbb R}
\newcommand{\N}{\mathbb N}
\newcommand{\ve}{\varepsilon}
\newcommand{\la}{\lambda}

\title{A general weak and strong error analysis of the  recursive quantization with an application to jump diffusions}

\author{ 
{\sc  Gilles Pag\`es} \thanks{Laboratoire de Probabilit\'es, Statistique  et Mod\'elisation (LPSM), Sorbonne Universit\'e, UMR  CNRS 8001, case 158, 4, pl. Jussieu, F-75252 Paris Cedex 5, France. E-mail: {\tt gilles.pages@upmc.fr} } \ \ \ 
{\sc  Abass Sagna} \thanks{ENSIIE \& Laboratoire de Math\'ematiques et Mod\'elisation d'Evry (LaMME), Universit\'e  d'Evry Val-d'Essonne,  UMR CNRS  8071,    23 Boulevard de France, 91037 Evry. E-mail: {\tt abass.sagna@ensiie.fr}. }\ \ \
\thanks{The first  author  benefited from the support of the Chaire ``Risques financiers'',  a joint initiative of \'Ecole Polytechnique, ENPC-ParisTech and Sorbonne Universit\'e, under the aegis of the Fondation du Risque. The second author  benefited from the support of the Chaire ``Markets in Transition'', under the aegis of Louis Bachelier Laboratory, a joint initiative of \'Ecole Polytechnique, Universit\'e d'\'Evry Val d'Essonne and  F\'ed\'eration Bancaire Fran\c{c}aise.} 
}

\date{}

\begin{document}

\maketitle

\begin{abstract}
Observing that the recent developments of the recursive (product) quantization method  induces a  family of Markov chains which  includes all standard discretization schemes of diffusions processes, we propose  to compute a general error bound induced by the recursive quantization schemes using this generic markovian structure.  Furthermore, we compute a marginal weak error for the recursive quantization. We also  extend the recursive quantization method to the Euler scheme associated to diffusion processes with jumps, which still  have this markovian structure, and  we say how to compute the recursive quantization and the associated weights and transition weights. 
\end{abstract}


\section{Introduction}

The $L^r$-optimal quantization problem for a random vector $X : (\Omega,\mathcal{A},\mathbb{P}) \longrightarrow  (\mathbb{R}^{d}, |\cdot|) $ at level $N$ consists on finding the best (w.r.t. the $L^r$-mean error) approximation of  $X$  by a Borel function   taking  at most  $N$ values. Assuming that $X \in L^r_{\mathbb R^d}(\mathbb P)$, this boils down to solve the following minimization problem 
 \begin{equation} \label{er.quant}
 e_{N,r}(X)   =  \inf{\{ \Vert X - \hat{X}^{\Gamma} \Vert_r, \  \Gamma \subset \mathbb{R}^d, \  \vert \Gamma \vert  \leq N \}}    =   \inf_{ \substack{\Gamma  \subset \mathbb{R}^d \\ \vert \Gamma \vert  \leq N}} \left(\int_{\mathbb{R}^d} {\rm dist}(x,\Gamma)^r d\mathbb P_X(x) \right)^{1/r} 
 \end{equation}
 where $|\Gamma|$ denotes the  cardinality of the set (or {\em grid}) $\Gamma$ and $\Vert X\Vert_r = \big[\E\, |X|^r\big]^{1/r}$ ($|\cdot|$ may be  a priori any norm on $\R^d$), $\PP_X$ denotes the distribution of $X$. The quantity  $\widehat X^{\Gamma}$ is  called a {\em Voronoi} or {\em nearest neighbour} quantization of $X$ on a grid $\Gamma = \{x_1, \ldots, x_N\} \subset  \mathbb R^d$, and is defined as
$\hat{X}^{\Gamma} = \sum_{i=1}^N  x_i  \mathds{1}_{\{X \in C_i(\Gamma)\}}$,  where $C_i(\Gamma)_{ i=1,\cdots,N}$  is  a Borel partition of  $\mathbb R^d$ (called a Voronoi partition   of $\mathbb{R}^d$)  satisfying for every $i \in \{1,\cdots,N\}$,  \ $ C_i(\Gamma) \subset  \big\{ x \in \mathbb{R}^d : \vert x-x_i \vert = \min_{j=1,\cdots,N}\vert x-x_j \vert \big\}$. The infimum in~\eqref{er.quant} is in fact a minimum i.e., for any  {\em level} $N\!\in \N$, there exists an optimal quantization grid $\Gamma^{(N)}$ solution to the above mimimization problem. For more insight on optimal quantization theory, we refer to~\cite{GraLus} or, for more applied topics,   to~\cite{PAGSpring2018}.

The quantization error $e_{N,r}(X)$ decreases to zero at the rate $N^{-1/d}$ as the grid size $N$ goes to infinity (this results is known as  the Zador Theorem, see e.g.~\cite{GraLus}).  There  is  also a non-asymptotic upper bound for  optimal quantizers called   Pierce Lemma and stated as follows in the quadratic case  (see~\cite{GraLus, LUPAaap}):    Let $p>2$. There exists  a universal constant  $C_{d,p, |\cdot|}$  such that  for every   random vector $X:(\Omega,{\cal A}, \mathbb P) \rightarrow (\mathbb{R}^{d}, |\cdot|)$,  
 \begin{equation}  \label{EqLemPierce} 
\inf_{\vert  \Gamma  \vert  \leq N} \Vert   X   - \widehat X^{\Gamma}  \Vert_{_2}   \le  C_{ d,p, |\cdot|} \, \sigma_{p}(X) N^{-\frac{1}{d}}
 \end{equation} 
 where
 $$  
 \sigma_{p}(X) = \inf_{\zeta \in \mathbb R^d} \Vert X-\zeta  \Vert_{_{p}} \le +\infty.
 $$
It will be briefly revisited in Section~\ref{subsec:regprodquant}.

\medskip From now on, $|\cdot|$  will always denote {\em the canonical Euclidean norm} $$|(y^1,\ldots,y^d)|= \Big(\sum_{1\le \ell \le d} (y^{\ell})^2\Big)^{1/2}.$$ Specific notations will be used to denote other norms on $\R^d$ (like $\ell^p$-norms).

\smallskip
For stochastic processes, the (fast) recursive quantization method has been introduced in~\cite{PagSagMQ} to quantize the Euler scheme associated to a Brownian  diffusion  process. To briefly recall the principle of recursive quantization let us  consider the Euler scheme  associated to the stochastic process solution of the stochastic differential equation
\begin{equation} \label{EqSignalIntro}
 X_t = X_0+\int_0^t b(s,X_s) ds + \int_0^t\sigma(s,X_s) dW_s,   
 \end{equation}
where  $W$ is a  standard  $q$-dimensional  Brownian motion independent from  $X_0 \in \mathbb R^d$, both defined on a probability space $(\Omega,{\cal A}, \mathbb P)$. The functions $b:[0,T] \times \mathbb R^d \rightarrow \mathbb R^d$ and the matrix-valued  diffusion coefficient function  $\sigma:[0,T] \times \mathbb R^d \rightarrow \mathbb R^{d \times q}$ are Borel measurable and satisfy some appropriate Lipschitz continuity  and linear growth conditions in $x$ uniformly in $t\!\in [0,T]$  to ensure  the existence of  a strong solution of  \eqref{EqSignalIntro}. We recall that for a given  regular time discretization mesh  $t_k = k  \Delta$, $k=0,\ldots, n$, $h = T/n$, the    Euler scheme  $(\bar  X_{t_k})_{k=0,\ldots,n}$, associated to  $(X_t)_{t\in [0,T]}$  is recursively  defined  by 
\begin{align} 
 \label{EqEuler} \bar X_{t_{k+1}}&= \bar X_{t_k} + h \,b(t_k,\bar X_{t_k})  + \sigma(t_k,\bar X_{t_k}) (W_{t_{k+1}} - W_{t_k}), \quad  \bar X_0 = X_0\\
\nonumber & =\bar X_{t_k} + h \,b(t_k,\bar X_{t_k})  + \sqrt{h} \sigma(t_k,\bar X_{t_k})Z_{k+1}\quad \mbox{ with }\quad Z_{k+1}= \frac{ W_{t_{k+1}} - W_{t_k}}{\sqrt{h}}\\
\label{eq:Eronde} & =: {\cal E}^h_k(\bar X_{t_k} ,Z_{k+1}) ,\; k=0,\ldots,n-1
\end{align} 
(hence  $(Z_k)_{k=1,\ldots, n}$ is an  i.i.d. sequence of  ${\cal N}(0;I_q)$-distributed random vectors, independent of  $\bar X_0$). 
 The  recursive  (marginal) quantizations $(\widehat X_{t_k}^{\Gamma_k}  )_{k=0,\ldots,n}$ (on the grids $(\Gamma_k)_{0 \le k \le n}$)  of  $(\bar X_{t_k}  )_{k=0,\ldots,n}$ are defined from the following recursion: 
 \begin{align}
\nonumber   \widetilde X_0  &=  {\rm Proj}_{\Gamma_0}(\bar X_0)\qquad \mbox{(if $X_0=x_0$ then $\Gamma_0= \{x_0\}$)} \\
\nonumber  \widehat X_{t_k}^{\Gamma_k}   &= {\rm Proj}_{\Gamma_k}(\widetilde X_{t_k})  \\
 \label{eq:rec} \quad \textrm{and}\quad  \widetilde X_{t_{k+1}} &= {\cal E}^h_k(\widehat X_{t_k}^{\Gamma_k} ,Z_{k+1}) ,\; k=0,\ldots,n-1,
  \end{align}
where ${\rm Proj}_{\Gamma}$ denotes a Borel nearest neighbor projection on  a finite grid $\Gamma\subset \R^d$. Furthermore, at each time step, the grid $\Gamma_k$ is $L^2$-optimal at a prescribed level $N_k$ i.e.
\begin{equation}\label{eq:OptiNewt}
\Vert \widehat X_{t_k}^{\Gamma_k}-\widetilde X_{t_k}\Vert_2 = \min \Big\{\big\Vert \widetilde X_k-{\rm Proj}_{\Gamma}(\widetilde X_{t_k})\big\Vert_2, \; |\Gamma|\le N_k\Big\}.
\end{equation}
where ${\cal E}_k^h$ is defined by~\eqref{eq:Eronde}. 
One of the main advantage of the method is that it may produce, in the particular one dimensional setting, the optimal marginal quantization grids of the Euler scheme $(\bar X_{t_k})_{k=0:n}$ and their associated probability weights quite instantaneously.  This follows from the  fact  that the recursion procedure in~\eqref{eq:OptiNewt} allows the use of the Newton algorithm~--~or possibly any fast deterministic optimization algorithm~--~to solve the grid optimization problem~\eqref{eq:OptiNewt} at every step of the algorithm.  Furthermore, under the above assumptions on $b$ and $\sigma$,  an error bound (valid in dimension $d$) for the quantization errors $\Vert \bar X_{t_{i}} -  \widehat X_{t_{i}}^{\Gamma_i}  \Vert_{_2}$ is  given in~\cite{PagSagMQ} where it is established that, for  every $k=0, \ldots, n$ 
 \begin{equation*} 
  \Vert  \bar X_{t_k} -  \widehat X_{t_k}^{\Gamma_k} \Vert_{_2}  \le    \sum_{\ell=0}^{k}  c_{\ell}  N_{\ell} ^{-1/d},
 \end{equation*}
where the $c_{\ell}$, $\ell=1:d$ are positive real constant depending on $b$, $\sigma$, $h$, $x_0$ and a parameter $p>2$ coming from Pierce's Lemma.

 The (regular) recursive quantization, as described previously,  cannot been efficiently implemented in  dimension $d\ge 2$ since  it relies on  the computation of $d$-dimensional optimal quantization grids requiring the use stochastic optimization algorithms instead of the Newton algorithm, making the procedure significantly more time consuming.  To overcome this issue, an (efficiently implementable) extension of the recursive quantization  to the $d$-dimensional  setting  has been proposed in~\cite{FioPagSag}.  It is based on a Markovian and componentwise product quantization  of the process $(\bar X_k)_{0 \le k \le n}$. To define precisely the method, let  $\ell \in \{ 1, \ldots, d\}$ and  let us denote by $\Gamma_k^{\ell} = \{x_k^{\ell,i_{\ell}}, \, i_{\ell}=1, \ldots,N_k^{\ell}\}$ an  $N_k^{\ell}$-quantizer  of   the $\ell$-th component  $\bar X_k^{\ell}$ of  $\bar X_k$. Denote by   $\widehat X_k^{\ell}$, the  quantization  of $\bar X_k^{\ell}$ of size $N_k^{\ell}$,  on the grid $\Gamma_k^{\ell}$.    We    define the  product quantizer $\Gamma_k =  \bigotimes_{\ell=1}^d  \Gamma_k^{\ell}$ of size  $N_k = N_k^1  {\small \times }  \ldots {\small \times }  N_{k}^d$ of  the vector $\bar X_k$  as   
  \begin{eqnarray*}
  \Gamma_k &  = &  \big\{ (x_k^{1,i_1}, \ldots,x_k^{d,i_d}),    \quad   \, x_k^{\ell,i_{\ell}} \in \Gamma_k^{\ell}\quad  \textrm{ for }\,  \ell \in \{1, \ldots,d\}    \, \textrm{  and } \, i_{\ell} \in \{1, \ldots,N_k^{\ell}\} \big\}.  
  \end{eqnarray*}
  
If we assume  that  $\bar X_0$ is already  quantized as $\widehat X_0$,  we define recursively the product quantization $(\widehat X_{t_k})_{0 \le k \le n}$ of the process $(\bar X_{t_k})_{0 \le k \le n}$  by the following procedure:
   \begin{equation}  \label{EqAlgorithmIntro}
   \left \{ \begin{array}{l}
 \widetilde X_0 = \widehat X_0, \quad  \widehat X_k ^{\ell} = {\rm Proj}_{\Gamma_k^{\ell}}(\widetilde X_k^{\ell}), \ \ell=1, \ldots, d   \\
   \widehat X_k  =  ( \widehat X_k^1, \ldots, \widehat X_k^d)  \quad  \textrm{and} \quad  \widetilde X_{k+1}^{\ell}  =  {\cal E}_k^{\ell}(\widehat X_k,Z_{k+1}), \ \ell=1, \ldots, d \\
   {\cal E}_k^{\ell} (x,z)= x^{\ell} + h b^{\ell}(t_k,x) + \sqrt{h} (\sigma^{\ell  \bullet}(t_k,x) \vert  z),  \ z = (z^1, \ldots,z^q)\in \mathbb R^q\\
     x=(x^1, \ldots,x^d), \ b=(b^1, \ldots, b^d)   \textrm{ and } (\sigma^{\ell  \bullet}(t_k,x) \vert  z) =\sum_{m=1}^q \sigma^{\ell m}(t_k,x) z^m
 \end{array}  
 \right.
 \end{equation}
 where for $a \in {\cal M}(d,q)$, $a^{\ell \bullet}  = [a_{\ell j}]_{j=1, \ldots,q}$.  
%
 
 \medskip
 From the numerical viewpoint,  higher order schemes (in particular, the Milstein scheme and the simplified weak order 2.0 Taylor scheme) are implemented in~\cite{McWalter} to improve the recursive quantization based Euler scheme. Strong error bonds directly adapted from~\cite{PagSagMQ} are established for the Milstein scheme in~\cite{RuddPhD}.
Recursive quantization is also applied to other model, like    local volatility models in~\cite{CalFioGra1} for calibration purposes,   to stochastic volatility models~\cite{CalFioGra2} (including Heston model). In the latter setting, in order to quantize the pair price-volatility process, the authors  first  quantize the volatility process which does not depend on the price process and then ``plug'' its quantization into the price process prior to   quantizing this second component of the pair. This appears as a particular case of the general {\em product  quantization method} of the Euler scheme (see Remark~\ref{RemForFioCal} in~\cite{FioPagSag}) developed to extend the recursive quantization paradigm to higher dimensional Brownian diffusions. 

One of our aim is then to give  recursive quantization error bounds  extending those  established  in~\cite{PagSagMQ} to a quite  general Markov framework, unifying on the way the previously cited works. 

\medskip
On the other hand, a first application of  recursive quantization  has been proposed   in~\cite{CalFioGra3} to  pure jump processes when  the characteristic function  of its marginal has an explicit expression or may be computed efficiently.  One of the contributions of this paper is to extend the recursive quantization to  a general jump diffusion solution to a stochastic differential equation driven by both a Brownian motion and a compound Poisson process evolving as
 \begin{equation} \label{EqJumpDif}
X_t = x + \int_0^t b(s,X_s) ds + \int_0^t \sigma(s,X_s) dW_s + \int_0^t  \gamma(X_{s-}) d\tilde \Lambda_s,
\end{equation}
 where $\tilde \Lambda  $ is a compensated compound  Poisson process defined w.r.t. a Poisson process $P$  by $\tilde \Lambda_t =  \sum_{i=1}^{\Lambda_t} U_{i} - \lambda t\, \E \, U_1$, for every $t \ge 0$, with $\Lambda$ a Poisson process with intensity $\lambda>0$. The sequence $(U_i)_{i \ge 1}$ is a i.i.d sequence (with distribution $\nu$)  of random variables, corresponding to the size of the jumps.  To be more precise, our aim is to recursively quantize the Euler scheme of this jump SDE defined from the following recursion, starting from $\bar X_0 = X_0$ by: 
\begin{eqnarray}
\bar X_{k+1} &=& \bar X_k +h\, b(t_k, \bar X_k) + \sqrt{h}\,\sigma(t_k, \bar X_{t_k})Z_{k+1}+ \gamma(\bar X_k) \big(\tilde \Lambda_{(k+1)h}-\tilde \Lambda_{kh}\big)  \label{EulerJump} \\
&
= & \bar X_{t_k} + h\,  b(t_k,\bar X_{t_k})  + \sqrt{h} \,\sigma(t_k,\bar X_{t_k})  Z_{k+1} + \gamma(\bar X_k) \Big(\sum_{\ell= \Lambda_{t_k}+1}^{h \Lambda_{t_{k+1}}}  U_{\ell}  - \lambda h\, \E\, U_1 \Big)
\end{eqnarray}
where  we consider a regular time discretization points $t_k =\frac{k T}{n}=k h$ for every $k\!\in \{ 0, \ldots,n \}$. 
 We will often consider  the classical modification where there is at most one jump, with probability $\lambda h$ (see Section~\ref{SecJump1} and Section~\ref{sec:NumerExp}). Like for the recursive quantization of Euler scheme of Brownian SDEs, this allows as to speak of {\em fast quantization} since the quantization of the whole path of the Euler process and its companions weights and transition probability weights may be computed  instantaneously  provided closed forms or fast algorithms are available for the quantization of  the distribution $\nu$  itself.

\medskip
We then observe that all the numerical schemes mentioned in this introduction share a Markov property
with similar features (among others, Lipschitz continuity propagation under natural assumptions), including  the  above extended   jump diffusions framework.  This lead us to propose and analyze a general  recursive quantization in  discrete time Markovian framework.

More precisely,  we suppose that a given scheme or more generally a discrete time Markov  process $\bar X = (\bar X_k)_{k=0:n}$  has the following generic form: $\bar X_0 \in \mathbb R^d$,
\[
\bar X_{k} = F_{k}\big(\bar X_{k-1}, Z_{k}\big), \; k=1:n,
\]
where $(Z_k)_{k=1:n}$ are i.i.d. $\R^q$-valued  random vectors defined on  a probability space $(\Omega,{\cal A}, \P)$ and $F_k:\R^d\times \R^q\to \R^q$, $k=1:n$ are   Borel functions. 
We define the recursive marginal quantization of $(\bar X_k)_{k=0:n}$  on the grids $\Gamma_k$, $k=0:n$, by  $\widehat X_0=   {\rm Proj}_{\Gamma_0}(\bar X_0)$, 
\begin{equation}  \label{EqGenHatX}
\widehat X_k = {\rm Proj}_{\Gamma_k}(\widetilde X_k), \; k=1:n, 
\end{equation}
where $(\widetilde X_k)_{k=1:n}$ is recursively defined by 
\begin{equation}  \label{EqGenTildeX}
\widetilde X_k = F_k\big(\widehat X_{k-1}, Z_k\big), \; k=1:n.
\end{equation}
At each time step, we assume that $\Gamma_k$ is an $L^2$-optimal grid for the distribution of $\widetilde X_k$. Supposing that the $F_k$'s are $[F_k]$-Lipschitz and have the following $L^p$-sub-linear growth  property for some $p \in (2,3]$, namely: for every $k\!\in \{0,\ldots, n\}$ and every $x \!\in \R^d, \quad  \E\vert F_k(x,Z) \vert^p\le \alpha_{p,k}+\beta_{p,k} |x|^p$, $\alpha_{p, k}, \beta_{p,k}\textcolor{blue}{\ge } 0$, we show that  the mean quadratic  recursive quantization error is given by
\[
\big\| \bar X_k-\widehat X_k\big\|_{_2} \le C_{d, p, \vert \cdot \vert } \, d^{\frac{p-2}{2p}} \sum_{i=0}^k   [F_{i+1:k}]_{\rm Lip}  \Big[ \sum_{\ell=0}^i  \alpha_{p,\ell}\beta_{p,\ell:i}\Big] ^{\frac 1p} \, N_i^{-\frac 1d}
\]
where $ [F_{i+1:k}]_{\rm Lip} $ and  $\beta_{p,\ell:i}$ are  constant depending on $\alpha_{p, \ell}$, $\beta_{p,\ell}$ and $[F_\ell]$ and will  be specified further on (see Proposition~\ref{PropGlobBound}) and where  $C_{p}>0$ only depends on  $p$ and $d$. Note that we will then specify in a more precise way these coefficients in all schemes under consideration in the paper. 

When using the recursive {\em product} quantization,  which consists, roughly speaking, in quantizing optimally in $L^2$ each {\em marginal} of $\widetilde X_k$ and consider as a grid  $\Gamma_k $ the product of these optimal marginals grids, the recursive quantization error becomes, under the same structure assumptions on the model  
\[
\big\| \bar X_k-\widehat X_k\big\|_{_2} \le  C_{p} \,d^{^{\frac{p-2}{2p}}}  \sum_{i=0}^k   [F_{i+1:k}]_{\rm Lip}  \Big[ \sum_{\ell=0}^i  \alpha_{p,\ell} \,\beta_{p,\ell:i} \,d^{^{(\frac{p}{2}-1)(i-\ell)} } \Big] ^{\frac{1}{p}} \, N_{i}^{-\frac 1d}
\]
 where $ C_{p}$ is  a positive real constant only depending on $p\!\in (2,3]$.  As expected,  there is, at least theoretically,  a loss of accuracy due to the presence of the factors $d^{^{(\frac{p}{2}-1)(i-\ell)}}$ in the right hand side of the above inequality,  whereas as detailed in~\cite{FioPagSag}, the numerical optimization of product grids can be performed in  very fast deterministic way whereas the computation  of the {\em regular} optimal grid $\Gamma_k$ in~\eqref{EqGenHatX} requires slower stochastic optimization procedures.

 We then give a general result (see Lemma \ref{lem:key} and~\cite{PagSagMQ}) stating  how to specify the coefficients $\alpha_{p, k}$, $\beta_{p, k}$ and $[F_k]_{\rm Lip}$  in various numerical schemes and models  under consideration in this paper (Euler, 1D-Milstein, simplified 2.0, Euler with jumps).

\medskip
We also provide a marginal  weak error associated to the recursive quantization. In fact,  under smooth conditions on the transition kernel induced  by the Markov chain $\bar X$,  we show  that for any function $f \in {\cal C}^1(\mathbb R^d,\mathbb R)$ such that $ [\nabla f]_{\rm Lip}<+\infty$ and $\forall\, \ell \ge 0, \  [\nabla P^{\ell} f]_{\rm Lip}<+\infty$,
 \begin{equation} \label{EqWeakErr}
 \big|\E\, f(\hat X_k)-\E\,f(\bar X_k)\big|\le \frac{[\nabla f]_{\rm Lip}}{2} \sum_{\ell=0}^k[\nabla P^{k-\ell} f]_{\rm Lip}\big\|\widehat X_\ell-\widetilde X_\ell\big\|^2_2.
 \end{equation}
We will provide, under appropriate assumptions, explicit bounds for $[\nabla P^{k-\ell}]_{Lip}$ (with controls depending on  the discretiation size $n$) for all the discretization schemes under consideration (Euler, Milstein, simplified 2.0, Euler with jumps), see Section~\ref{SectionWeakError}.

 \medskip

 The paper is organized as follows. In Section \ref{Section1}, we make a general  error analysis of the strong error of the recursively quantized scheme  for both  regular and  product recursive  quantization methods. We then deduce the recursive quantization error bounds associated to some usual schemes like the Euler scheme (for jump and no jump diffusions), the Milstein scheme, the simplified weak order 2.0 Taylor scheme. In Section \ref{SectionWeakError} , we address a first weak error analysis for the recursive quantization by giving  a general result followed by  specific results associated to some  schemes. Section~\ref{sec:NumerExp} is devoted to more algorithmic developments about the recursively quantized Euler scheme of a diffusion. We give in this section a numerical example which compares the quantization distributions of a no jump and a  jump diffusion process with normally  distributed jump sizes.  We also give a numerical example for the pricing of  a put in jump model to test the performances of the recursive quantization for jump diffusions.


\section{General recursive quantization error analysis for Markov dynamics} \label{Section1}

As pointed out in the introduction, the  recursive quantization  
 procedure induces a sequence $(\widehat X_k)_{0\le k \le n}$ of quantizations which has a generic form including the specific procedures in all the previously indicated papers. To setup the general framework,   let  us consider an $\R^d$-valued  Markov chain  $(\bar X_k)_{k=0:n}$, defined as an iterated mapping of the form
\begin{equation}  \label{eq:Markov}
\bar X_{k} = F_{k}\big(\bar X_{k-1}, Z_{k}\big), \; k=1:n,
\end{equation}
where $(Z_k)_{k=1:n}$ are i.i.d. $\R^q$-valued  random vectors defined on  a probability space $(\Omega,{\cal A}, \PP)$ and $F_k:\R^d\times \R^q\to \R^q$, $k=1:n$ are   Borel functions. Hence the transitions $P_k(x,dy)= \PP(\bar X_{k+1}\!\in dy \,|\, \bar X_k=x)$ of $(\bar X_k)_{k=0:n}$ read on Borel functions $f:\R^d \to \R$,
\[
P_kf(x)= \E\, f\big(F_{k+1}\big(x,Z_{k+1}\big)\big), \; \qquad x\!\in \R^d.
\] 
Such a family of Markov chains includes all standard discretization schemes of diffusions with or without jumps.

\medskip
We define a recursive marginal quantization (in fact Markovian) of $(\bar X_k)_{k=0:n}$  on grids $(\Gamma_k)_{k=0:n}$ by  $\widehat X_0=   {\rm Proj}_{\Gamma_0}(\bar X_0)$, 
\begin{equation}  \label{EqGenHatX}
\widehat X_k = {\rm Proj}_{\Gamma_k}(\widetilde X_k), \; k=1:n, 
\end{equation}
where $(\widetilde X_k)_{k=1:n}$ is recursively defined by 
\begin{equation}  \label{EqGenTildeX}
\widetilde X_k = F_k\big(\widehat X_{k-1}, Z_k\big), \; k=1:n.
\end{equation}
In the recursive quantization procedure, we quantize in fact $\widetilde X_k$ as $\widehat X_k$ at every step $k$, $k=1, \ldots, n$, of the algorithm, supposing that the initial r.v. $\bar X_0$ may be quantized as $\widehat X_0$. The question of interest is  to compute the quadratic quantization error induces by such a procedure, means, the quantity $\Vert  \bar X_k - \widehat X_k  \Vert_2$. To this end, we need to make the following main assumptions.

\medskip

\noindent {\it Main assumptions}.  We consider  the following two main assumptions. 
\begin{enumerate}
\item We suppose that the functions $F_k$  and (the   distribution of) $Z=Z_1$ is $L^2$ Lipschitz continuous:  
\[
({\rm Lip})\quad \equiv \quad \big\|F_k(x,Z)-F_k(x',Z)\big\|_{_2}\le [F_k]_{\rm Lip} |x-x'|, \; x,\, x' \!\in \R^d,\; k=1:n.
\]  

\item For   $\, p\!\in (2,3]$, we introduce  the following $L^p$-sub-linear growth assumption on the functions $F_k$
\[
({\rm SL})_p\; \equiv  \; \forall\, k\!\in \{0,\ldots, n\},\; \forall\,x \!\in \R^d, \quad  \E\vert F_k(x,Z) \vert^p\le \alpha_{p,k}+\beta_{p,k} |x|^p.
\]
\end{enumerate}

We will compute the recursive quantization error under assumptions  $({\rm Step})$ and $({\rm Lip})$.  When we consider  a diffusion process (possibly with jumps),  this former assumption depends on the used discretization scheme, more precisely, on the coefficients $\alpha_{p,k}$ and $\beta_{p,k}$ which  are deduced from the control of the transition operator of the  considered discretization scheme and on its Lipschitz coefficients $[F_k]_{\rm Lip}$.  In fact,  in the general setting, $F_k(x, Z_{k})$ may be decomposed as $F_{k}(x, \zeta) = a(x) + \sqrt{h} A (x)\zeta$, where $a: \mathbb R^d \rightarrow  \mathbb R^d$, $A: \mathbb R^d \rightarrow {\cal M}_{d,q}(\mathbb R)$ is a $d \times q$ matrix valued function and $\zeta$ is a centered random variable.    The  lemma below gives a control of the  generic form of the transition operator induced by the usual discretization schemes (including the Euler scheme, the Milstein scheme, etc) associated to a diffusion process (with or without jumps).  The proof of this key lemma follows the proof of Lemma 3.1. in~\cite{PagSagMQ} and   is postponed in the appendix.  
 
 In the statement of the following lemma and in the rest of the paper, we will denote the  $\mathbb R^d$-valued function  $a(x)$ by $a$ and the $ {\cal M}_{d,q}(\mathbb R)$-valued function $A(x)$ by $A$ to alloviate the notations.

 \begin{lem}[Key lemma]  \label{lem:key} $(a)$ Let $A$ be a $d {\small \times} q$-matrix and let $a\!\in \R^d$.  Let $p\!\in [2,3)$. For any random vector $\zeta$ such that  $\zeta\!\in L_{\R^q}^p(\Omega,{\cal A}, \mathbb P) $ and  $\mathbb E \,\zeta =0$, one has for every $h\!\in (0,+\infty)$
\[
\mathbb E  \vert a + \sqrt{h} A \zeta \vert^p   \le    \Big(1+  \frac{(p-1)(p-2)}{2}  h \Big) \vert a \vert^p   +  h \Big( 1 + p + h^{\frac{p}{2}-1} \Big)  \Vert A  \Vert^p \, \mathbb E \vert  \zeta \vert^p,
\]
where   $\Vert A \Vert^2  = {\rm Tr} (A A^{\star})$. 

\medskip
\noindent $(b)$ In particular, if $|a|\le |x|(1+Lh)+Lh$ and $\|A\|^p\le 2^{p-1}\Upsilon^{^p}h (1+|x|^p)$, then 
 \begin{equation}   \label{EqKeyLem}
   \mathbb E  \vert a + \sqrt{h} A \zeta \vert^p   \le     \Big( e^{ \kappa_p h}  + K_p h  \Big)  \vert x \vert^p  +  \big(  e^{\kappa_p  h } L + K_p \big) h, 
   \end{equation} 
where $\kappa_p := \Big(\frac{(p-1)(p-2)}{2 } + 2 p L \Big)$  and $K_p :=  2^{p-1} \Upsilon^{^p} \Big( 1 + p + h^{\frac{p}{2}-1} \Big)   \mathbb E \vert  \zeta \vert^p$.
\end{lem}

\begin{rem} \label{RemKeyLem}  It follows from Lemma \ref{lem:key}, more particularly from Equation  \eqref{EqKeyLem}, that if   $F_k$ has the generic form $F_{k}(x, \zeta) = a(x) + \sqrt{h} A \zeta$, with  $|a(x)|\le |x|(1+Lh)+Lh$ and $\|A\|^p \le 2^{p-1} \Upsilon^{^p} h (1+|x|^p)$ then we may choose $\alpha_{p,k} = \big(  e^{\kappa_p  h } L + K_p \big) h$ and $\beta_{p,k} =  e^{ \kappa_p h}  + K_p h$.

\end{rem}

\subsection{Regular recursive quantization}\label{subset:regrecquant}
The following result gives a general quadratic quantization error bound associated to the standard  recursive quantization and according to  the coefficients $\alpha_{p,k}$, $\beta_{p,k}$ and  $F_k$.

\begin{prop} \label{PropGlobBound} Let $\bar X$ and $\widehat X$ be defined by \eqref{EqGenHatX} and \eqref{EqGenTildeX} and suppose that both assumptions $({\rm Lip})$ and  $({\rm SL})_p$ (for some $p \in (2,3]$)  hold. Then, 
\begin{equation}  \label{EqRecQuantBound}
\big\| \bar X_k-\widehat X_k\big\|_{_2} \le C_{d, p} \sum_{i=0}^k   [F_{i+1:k}]_{\rm Lip}  \left[ \sum_{\ell=0}^i  \alpha_{p,\ell}\beta_{p,\ell:i}\right] ^{\frac 1p} \, N_i^{-\frac 1d}
\end{equation}
where $C_{d, p}\le C_{p, d,  \vert \cdot \vert}\, d^{\frac{p-2}{2p}}$ ($C_{p, d, \vert \cdot \vert}$ is the  universal constant appearing in the  Pierce Lemma, see~\eqref{EqLemPierce} and Lemma~\ref{PierceExt} later on),  $\beta_{p,\ell:k}=\prod_{m=\ell+1}^{k}\beta_{p,m}$,  with the convention $\alpha_{p,0}=\Vert  \bar X_0 \Vert^{^p}_{_p}$ and $\prod_{\emptyset}= 1$, and where
\[
 [F_{k+1:k}]_{\rm Lip}=1  \qquad \mbox{ and } \quad   [F_{\ell:k}]_{\rm Lip}:=  \prod_{i=\ell+1}^k [F_{i}]_{\rm Lip},\; 0\le i\le  k\le n.
\]

\end{prop}

\begin{proof} {\em First step}.  We have, for every $k\!\in \{0,\ldots,n-1\}$ 
\begin{eqnarray*}
\widehat X_{k+1}-\bar X_{k+1} & \, = \, &\widehat X_{k+1}-\widetilde X_{k+1} + \widetilde X_{k+1}- \bar X_{k+1}\\
&\, =\, & \widehat X_{k+1}-\widetilde X_{k+1} + F_{k+1}(\widehat X_k, Z_{k+1})- F_{k+1}( \bar X_k, Z_{k+1})
\end{eqnarray*}
by the very definition of the sequences $(\bar X_k)_k$ and $(\widetilde X_k)_k$. Hence,
\[
\big\| \widehat X_{k+1}-\bar X_{k+1} \big\|_{_2} \le \big\| \widehat X_{k+1}-\widetilde X_{k+1} \big\|_{_2} + [F_{k+1}]_{\rm Lip}  \big\| \bar X_k-\widehat X_k\big\|_{_2}
\]
owing to  Assumption   $({\rm Lip})$. A straightforward induction shows that, for every $k\!\in \{0, \ldots,n\}$
\begin{equation}\label{eq:errorbound}
\big\|\bar X_k-\widehat X_k\big\|_{_2} \le \sum_{\ell=0}^k   [F_{\ell:k}]_{\rm Lip} \big\|\widehat X_\ell-\widetilde X_\ell\big\|_{_2}.
\end{equation}

On the other hand,  if we assume that all the grids $\Gamma_k$ are $L^2$-optimal, then it follows from the extended Pierce Lemma (see Equation~\eqref{EqLemPierce}) that,  for every $k\!\in \{0, \ldots,n\}$, 
\begin{eqnarray}
\big\vert  \widehat X_{k}-\widetilde X_{k} \big\vert_{_2} & = & \big\vert {\rm Proj}_{\Gamma_k}(\widetilde X_k)- \widetilde X_k \big\vert_{_2} \nonumber \\
&\le& C_{d,p, \vert \cdot \vert}\, \sigma_{p}(\widetilde X_k) \, |  \Gamma_k|^{-\frac 1d}  \label{eq:Pierce}
\end{eqnarray}
where $\sigma_{p}(Y)= \inf_{a\in \R^d} \|Y-a\|_{_{p}}\le \|Y\|_{_{p}}$. 

\medskip

\noindent {\em Second step}. The next step is to control this pseudo-standard deviation terms $\sigma_{p}(\widetilde X_k)$. In fact, we will simply upper-bound $\|\widetilde X_k\|_{p}$.
Using again that the grids $\Gamma_k$ are $L^2$-optimal (and then, stationary), we get
\begin{eqnarray*}
\E\, \vert \widehat X_{k+1} \vert^p 
= \E\, |\E\big(\widetilde X_{k+1}\,|\, \widehat X_{k+1}\big)|^p \le  \E\,| \widetilde X_{k+1}|^p,
\end{eqnarray*}
owing to Jensen's Inequality. On the other hand, using this time Assumption~$({\rm SL})_p$ yields
\begin{eqnarray*}
\E \big|\widetilde X_{k+1}\big|^p &= & \E\,\big|F_{k+1}(\widehat X_k, Z_{k+1})\big|^p\\
&=& \E\,\big( \E( |F_{k+1}(\widehat X_k, Z_{k+1})|^p\,|\,\widehat X_k )\big)\\
&\le& \E\big(\alpha_{p,k+1}+\beta_{p,k+1}|\widehat X_k|^p\big)  
\end{eqnarray*}
since $\widehat X_k$ and $Z_{k+1}$ are independent. We use the stationarity property of the quantizers $\widehat X_k$ and Jensen inequality  to get
\begin{eqnarray*}
\E \big|\widetilde X_{k+1}\big|^p &\le& \E\big(\alpha_{p,k+1}+\beta_{p,k+1}  | \E (\widetilde X_k | \widehat X_k) |^p\big) \\
&= & \E\big(\alpha_{p,k+1}+\beta_{p,k+1}\E (|\widetilde X_k|^p | \widehat X_k )\big) \\
  &= & \E\big(\alpha_{p,k+1}+\beta_{p,k+1} |\widetilde X_k|^p | \big).
\end{eqnarray*}
Then, we derive by a standard induction (discrete time Gronwall Lemma) that, 
\begin{equation}  \label{EqBoundXkp}
\E\,\big| \widetilde X_{k}\big|^p\le \sum_{\ell=0}^k \alpha_{p,\ell}\,\beta_{p,\ell:k}
\end{equation}
where $\beta_{p,\ell:k}=\prod_{m=\ell+1}^{k}\beta_{p,m}$
 and  with the convention $\alpha_{p,0}=\big\|  \bar X_0\big\|^p_{_p}$ and $\prod_{\emptyset}= 1$. 

One concludes by plugging this bound  into~\eqref{eq:Pierce} and then in~\eqref{eq:errorbound} which yields
\[
\big\| \bar X_k-\widehat X_k\big\|_{_2} \le C_{d, p} \sum_{i=0}^k   [F_{i+1:k}]_{\rm Lip}  \left[ \sum_{\ell=0}^i  \alpha_{p,\ell}\beta_{p,\ell:i}\right] ^{\frac 1p} \, N_i^{-\frac 1d}
\]
where $N_k = |\Gamma_k\big|$, $k=0:n$.
\end{proof}

In general this approach   cannot been efficiently implemented in  dimension $d\ge 2$ since  it requires to compute an multidimensional optimal grid. Stochastic optimization procedures that should be called upon for that purpose are time consuming. This leads us to introduce the {\em recursive product quantization}  introduced in~\cite{FioPagSag} in a Brownian diffusion framework and for which we propose an analysis in full generality in the next subsection.

\subsection{Recursive product  quantization  and revisited Pierce's lemma} \label{subsec:regprodquant}
Let us briefy recall what recursive product quantization is. We refer to~\cite{FioPagSag} for more details. Consider  the $\mathbb R^d$-valued diffusion process  $(X_t)_{t\in [0,T]}$ defined by \eqref{EqSignalIntro} et let    $(\bar  X_{t_k})_{k=0,\ldots,n}$ be the associated Euler scheme process (with regular discretization step $h = T/n$) defined  by $\bar X_0 = X_0$ and 
\begin{equation}  \label{EqEulerPQ}
\bar X_{t_{k+1}}= \bar X_{t_k} + b(t_k,\bar X_{t_k}) h  + \sigma(t_k,\bar X_{t_k}) \sqrt{h}\, Z_{k+1}, \qquad  Z_{k+1} \sim {\cal N}(0; I_d),
\end{equation} 
where  $b:[0,T] \times \mathbb R^d \rightarrow \mathbb R^d$ and   $\sigma:[0,T] \times \mathbb R^d \rightarrow \mathbb R^{d \times q}$.  We recall that the recursive product quantization of the process $\bar X$ is defined by the recursion \eqref{EqAlgorithmIntro}. So that  for every $k \ge 0$, we define the recursive product quantization of $\bar X_k$ as $\hat {X}_k = (\widehat X_k^1, \ldots, \widehat X_k^d)$, where each  $\widehat X_k^{\ell}$ is the recursive quantization of the $\ell$-th component $\bar X_k^{\ell}$ of the vector $\bar X_k$
 and is defined  by $\widehat X_k^{\ell}  = {\rm Proj}_{\Gamma_k^{\ell}}(\widetilde X_k^{\ell}),$ with
\begin{eqnarray*}
\widetilde X_{k}^{\ell}  &=&   {\cal E}_{k-1}^{\ell} ( \widehat X_{k-1},Z_k) \\
&= & \widehat X_{k-1}^{\ell} + h b^{\ell}(t_{k-1},\widehat X_{k-1}) + \sqrt{h} (\sigma^{\ell  \bullet}(t_{k-1},\widehat X_{k-1}) \vert  Z_k),  \ Z_k \sim {\cal N}(0; I_d).
\end{eqnarray*}
Now, set 
\begin{equation}  \label{EqDefProdQua}
\widetilde X_k  =  \begin{pmatrix}  \widetilde X^1_k \\ \vdots \\  \widetilde X^d_k   \end{pmatrix} \quad \mbox{ and } \quad   \bar X_{k}  = {F}_{k} (\bar X_{k-1}, Z_{k}) := \begin{pmatrix} F_k^{1}(\bar X_{k-1},   Z_k) \\   \vdots  \\ F_k^{d}(\bar X_{k-1}, Z_k) \end{pmatrix}
\end{equation}
where
 \[
F_k^{\ell}(x, z) =  x^{\ell} + h b^{\ell}(t_k,x) + \sqrt{h} (\sigma^{\ell  \bullet}(t_k,x) \vert  z),  \ z = (z^1, \ldots,z^q)\in \mathbb R^q.
\]
In this recursive product quantization framework  we  need to extend the results of Proposition \ref{PropGlobBound}. This is done  in Proposition~\ref{PropExtGlobBound} below.    The proof that follows is nothing but the second step of proof of the extended Pierce lemma,  that extends scalar Pierce's Lemma from $1$ to $d$ dimensions. It is reproduced for the reader's convenience. As stated it emphasizes the dependence of the constant $C_{d,p}$ in the dimension $d$.

\medskip

To establish the error bounds for the strong error in full generality, we need to revisit Pierce's Lemma to enhance the role played by the product quantization. Let $Y= (Y^1,\ldots,Y^d):(\Omega,{\cal A}, \PP)\to \R^d$ be a $d$-dimensional random vector. Let us recall how product quantization is defined: for every $\ell\!\in \{1,\ldots,d\}$, let $\widehat Y^{\ell, \Gamma_\ell}$ denote a scalar Voronoi (following the nearest neighbour rule) quantization of $Y^{\ell}$ induced by a finite grid $\Gamma_{\ell}\subset \R$. Then the product quantization of $Y$ by the product grid $\Gamma = \Gamma_1\times\cdots\times \Gamma_d\subset \R^d$ is defined by $\widehat Y^{\Gamma}= \big(\widehat Y^{1,\Gamma_1},\ldots, Y^{d,\Gamma_d}\big)$. One easily checks that $\widehat Y^{\Gamma}$ is a Voronoi quantization of $Y$ induced by $\Gamma$ with respect to any $\ell^p$-norm (or pseudo-norm), $0<p<+\infty$.

The revisited version of Pierce lemma below deals with this product quantization when $\R^d$ is equipped with the $\ell^r$-norm (or pseudo-norm) and the mean quantization error is measured in $L^r(\PP)$. 


\begin{lem}[Revisited Pierce's lemma]  \label{PierceExt} Let $p>r>0$. Assume  that $\R^d$ is endowed with the $\ell^r$-norm $|y|_{_{\ell^r}} =\big( Ê\sum _{i=1}^d|y^{\ell}|^r\big)^{1/r}$ and that $\|\cdot\|_r$ is defined accordingly. 

There exists a real constant $C_{p}\!\in (0,+ \infty)$ such that, for every  random vector $Y= (Y^1,\ldots,Y^d):(\Omega,{\cal A}, \PP)\to \R^d$ and every integer (or level)  $N\ge 1$, 
 \begin{equation}  \label{EqLemPierceExt} 
\inf\left\{\big\|Y-\big(\widehat Y^{1,\Gamma_1},\ldots, Y^{d,\Gamma_d}\big)\big\|_r, \quad \prod_{1\le \ell\le d}|\Gamma_\ell|  \le N\right\}  \le  C_{p} \, d^{\frac 1 r-\frac 1p} \sigma_{p}(Y) N^{-\frac{1}{d}}
 \end{equation} 
 where, for every $p\!\in (0, +\infty)$,  
 $  \sigma_{p}(Y) = \inf_{\zeta \in \mathbb R^d} \Vert Y-\zeta  \Vert_{_{p}} \le +\infty$ denotes the pseudo-$L^p$-standard deviation of $Y$.
\end{lem}

This reformulation says that the universal non-asymptotic bound provided by Pierce's Lemma is obtained by product quantization with the above constants. The above infimum is in fact a minimum since, if $Y\!\in L^r(\PP)$, every component can be optimally quantized by an optimal grid $\Gamma_{\ell}^{(r)}$. This follows from the fact that
\[
 \Vert Y   - \widehat Y \Vert_{_r}^r  = \sum_{\ell=1}^d   \Vert  Y^{\ell}   - \widehat Y^{\ell}  \Vert_{_r}^r
\]

\begin{proof}For simplicity set $N_{\ell}=|\Gamma^{\ell}|$, $\ell=1,\ldots,d$. When $d=1$ the above statement si simply the standard one dimensional Pierce Lemma (see~\cite{LUPAaap, PAGSpring2018}).    It follows from this one dimensional Pierce lemma  that
\begin{align*}
 \Vert Y   - \widehat Y  \Vert_{_r}^r  &= \sum_{\ell=1}^d   \Vert Y^{\ell}   - \widehat Y^{\ell}  \Vert_{_r} ^r  \le C_{1, p}^{r}\, \sum_{\ell=1}^d \big(N_{\ell}\big)^{-r}  \sigma_{p}^r (Y^{\ell}) \\
 &\le C_{1, p}^r \,  N^{-\frac{r}{d}} \sum_{\ell=1}^d  \sigma_{p}^r (Y^{\ell}).
\end{align*}
Now,  for any  $(a_1, \ldots, a_d) \in \mathbb R^d$, 
\begin{align*}
\sum_{\ell=1}^d  \sigma_{p}^r (Y^{\ell}) & \le  \sum_{\ell=1}^d  \Big(\mathbb E \vert Y^{\ell} - a_i \vert^{p}\Big)^ \frac{r}{p} \\
&\le   d^{^{\frac{p-r}{p}}}\Big( \sum_{\ell=1}^d  \mathbb E \vert Y^{\ell} - a_i \vert^{p}\Big)^ \frac{r}{p} \\
& =  d^{^{\frac{p-r}{p}}}\Big(  \mathbb E\big \vert Y - (a_1, \ldots,a_d) \big\vert_{_{\ell^{p}}}^{p}\Big)^ \frac{r}{p},
\end{align*}
the second inequality coming from the H\"{o}lder inequality. Now, using the inequality  $|\cdot |_{_{\ell^{p}}} \le \vert \cdot \vert_{_{\ell^r}}$, we deduce that
\begin{equation*}
 \Vert Y   - \widehat Y  \Vert_{_r}^r  \le   C_{1, p}^r  N^{-\frac{r}{d}} \, d^{^{\frac{p-r}{p}}}\Big(  \mathbb E\big \vert Y - (a_1, \ldots,a_d) \big\vert^{p}\Big)^ \frac{r}{p}.
\end{equation*}
The result follows by taking the $r$-th root on both sides of the previous inequality, then the infimum over $(a_1, \ldots, a_d) \in \R^d$ and setting $C_{p}= C_{1,p}$. 
\end{proof}

\begin{rem} In fact the above proof is only revisiting the original proof of  Pierce Lemma i from~\cite{LUPAaap} (and stated in its final form in~\cite{PAGSpring2018}). To our best knowledge, this proof in higher dimension always relies on a product quantization argument so that the established bound holds for product quantization as emphasized above.
\end{rem}

%
%

We are now in position to give the result on the quadratic error bound of the recursive product quantization.

\begin{prop} \label{PropExtGlobBound}  Let $\bar X$ and $\widehat X$ be defined by \eqref{EqDefProdQua} and suppose that the assumptions $({\rm Lip})$ and  $({\rm SL})_{p}$ hold for some $p\!\in (2,3]$. Then,
\begin{equation}  \label{EqRecQuantBound}
\big\| \bar X_k-\widehat X_k\big\|_{_2} \le  C_{p} \,d^{^{\frac{p-2}{2p}}}  \sum_{i=0}^k   [F_{i+1:k}]_{\rm Lip}  \left[ \sum_{\ell=0}^i  \alpha_{p,\ell} \,\beta_{p,\ell:i} \,d^{^{(\frac{p}{2}-1)(i-\ell)} } \right] ^{\frac{1}{p}} \, N_{i}^{-\frac 1d}
\end{equation}
where $ C_{p}$ is  a positive real constant. 
%
\end{prop}

\begin{rem} Before dealing with the proof,  remark that  we may deduce from the upper bound in \eqref{EqRecQuantBound}  that 
\begin{equation*}
\big\| \bar X_k-\widehat X_k\big\|_{_2} \le  C_{p} \,d^{^{(\frac{p}{2}-1) \big(k + \frac{1}{p}\big)} }  \sum_{i=0}^k   [F_{i+1:k}]_{\rm Lip}  \left[ \sum_{\ell=0}^i  \alpha_{p,\ell} \,\beta_{p,\ell:i}  \right] ^{\frac 1p} \, N_{i}^{-\frac 1d}.
\end{equation*}
This suggests  that when the dimension increases,  the recursive product quantization introduces the additional  factor $ d^{^{(-1 + p/ 2) (k +  1/ p )} } $ 
 with respect to the regular recursive quantization method.   
\end{rem}

\begin{proof}  First,  we have to keep in mind that in this framework, the whole  vector $\widehat X_k$ is no longer stationary but each of its components still be stationary. 

Recall from \eqref{eq:errorbound} that we have:
\begin{equation*}
\big\|\bar X_k-\widehat X_k\big\|_{_2} \le \sum_{i=0}^k   [F_{i:k}]_{\rm Lip} \big\|\widehat X_{i}-\widetilde X_{i}\big\|_{_2}.
\end{equation*}
Now, using Lemma \ref{PierceExt} (the revisited Pierce's lemma) with $r=2$, yields
\begin{eqnarray} \label{EqStep1}
\big\|\bar X_k-\widehat X_k\big\|_{_2} \le  C_{p} \,d^{^{\frac{p-2}{2p}}}  \sum_{i=0}^k   [F_{i:k}]_{\rm Lip}     \big \Vert  \widetilde X_i \big\Vert_{p}  N_i^{-\frac{1}{d}} . 
\end{eqnarray}
On the other hand,  owing to Jensen's Inequality and to the stationary property (see e.g.~\cite{GraLus} for further details on the stationary property)  which states in particular that $\E\, \big(\widetilde X_{k+1}^{\ell}\,|\, \widehat X_{k+1}^{\ell}\big) = \widehat X_{k+1}^{\ell}$ satisfied by each $\widehat X_{k+1}^{\ell}$  since each quantization of the marginal $\widetilde  X_{k+1}^{\ell}$ is $L^2$-optimal,  we have
for any $k \in \{0, \ldots, n\}$,
\begin{eqnarray*}
\E\, \vert \widehat X_{k} \vert^p 
 = \mathbb E \Big(\sum_{\ell=1}^d \vert  \widehat X_{k} ^{\ell} \vert^2 \Big)^{p/2} &\le&  d^{^{\frac{p}{2}-1}} \, \sum_{\ell=1}^d \E \,\big \vert \widehat X_{k}^{\ell} \big\vert ^{^p} \\
 &\le &   d^{^{\frac{p}{2}-1}} \, \sum_{\ell=1}^d  \, \E\, \big \vert \E\big(\widetilde X_{k}^{\ell}\, \vert\, \widehat X_{k}^{\ell}\big) \big \vert^{^p} \\
 & \le & d^{^{\frac{p}{2}-1}} \, \sum_{\ell=1}^d  \E\,| \widetilde X_{k}^{\ell}|^p \, = \, d^{^{\frac{p}{2}-1}} \, \mathbb E \Vert  \widetilde X_k \Vert_{_{\ell^p}}^{^p}. 
\end{eqnarray*}
Using the inequality  $\vert \cdot \vert_{_{\ell^{p}}} \le  \vert \cdot \vert_{_{\ell^{2}}}= \vert \cdot \vert$ yields 
\[
\E\, \vert \widehat X_{k} \vert^p  \le  d^{^{\frac{p}{2}-1}}  \mathbb E \vert  \widetilde X_k \vert^{^p}.
\]
 Now, using  Assumption~$({\rm SL})_p$ yields, for any $\ell \in \{1, \ldots,d\}$, 
\begin{eqnarray*}
\E \big|\widetilde X_{k+1} \big|^p =   \E\,\big( \E( |F_{k+1}(\widehat X_k, Z_{k+1})|^p\,|\,\widehat X_k )\big) &\le& \E\big(\alpha_{p,k+1}+\beta_{p,k+1}|\widehat X_k|^p\big)   \\
& \le & \alpha_{p,k+1}+\beta_{p,k+1} \, d^{\frac{p}{2}-1}  \E |\widetilde X_k|^p. 
\end{eqnarray*}
We deduce by a standard induction that for every $k \in \{0, \ldots, n\}$,
\begin{equation}  \label{EqBoundXkpp}
\E\,\big| \widetilde X_{k}\big|^p\le \sum_{\ell=0}^k \alpha_{p,\ell}\,\beta_{p,\ell:k}  \,d^{^{(\frac{p}{2}-1)(k - \ell)}}. 
\end{equation}
We conclude by  replacing  $ \Vert  \widetilde X_i \Vert_{_{p}} $ in \eqref{EqStep1} by its value using \eqref{EqBoundXkpp}.
\end{proof}

\subsection{Toward time discretization schemes}
At this stage, having in mind time discretization schemes, one may try to control all these bounds as a function of $n$ and of the total quantization budget $N= N_0+\cdots+N_n$. In that spirit we may assume that 
\begin{equation}
({\rm Step})\;\equiv\;\left\{\begin{array}{ll}
(i) & \forall\, k\!\in \{0,\ldots,n\},\, \;[F_k]_{\rm Lip} \le 1+ C_0\frac{T}{n},\\ &\\
(ii) &\, \forall\, k\!\in \{0,\ldots,n\},\;\alpha_{p,k}\le C_1\frac{T}{n}\; \mbox{ and }\; \; \beta_{p,k} \le 1+ C_2\frac{T}{n}.
\end{array}\right.
\end{equation}

Usually in a time discretization framework discrete time instants $k$ stand for absolute time $t_k= \frac {kT}{n}$ so that, under the above asumption,
\[
  [F_{\ell:k}]_{\rm Lip} \le \Big(1+ C_0\frac{T}{n}\Big)^{t_k-t_\ell}\le e^{C_0(t_k-t_\ell)}\quad\mbox{  and}\quad \beta_{p,\ell:k}\le e^{C_2(t_k-t_\ell)}
\]
 so that 
 \begin{eqnarray*}
 \sum_{\ell=0}^k  \alpha_{p,\ell}\beta_{p,\ell:k}&\le &e^{C_1T}\big\| \bar X_0\big\|_{_p}+\frac{C_1T}{n} e^{C_2\frac Tn}\frac{e^{C_2t_k}-1}{e^{C_2\frac Tn}-1}\\
 &\le& e^{C_1T}\big\| \bar X_0\big\|_{_p}+\frac{C_1}{C_2} e^{C_2\frac Tn}(e^{C_2t_k}-1).
\end{eqnarray*}
Finally, for every $k=0:n$, 
\begin{equation}  \label{EqBoundStep}
\big\| \bar X_k-\widehat X_k\big\|_{_2} \le C_{d, \eta} \sum_{i=0}^k   e^{C_0(t_k-t_i)}  \left[ e^{C_1T}\big\| \bar X_0\big\|_{_p}+\frac{C_1}{C_2} e^{C_2\frac Tn}(e^{C_2t_k}-1)\right] ^{\frac 1p} \, N_i^{-\frac 1d}.
\end{equation}

\subsection{Few examples of schemes} 
We now move towards the examples of schemes. Our aim in this step is to identify explicitly the coefficients $\alpha_{p,k}$, $\beta_{p,k}$ and the Lipschitz coefficients $[F_k]_{\rm Lip}$ for each given scheme. 

\subsubsection{Euler scheme (for both the regular and the product recursive quantization)} 
We consider here a one-dimensional Brownian diffusion with drift $b$ and diffusion coefficient $\sigma$ driven by a one-dimensional Brownian motion $W$  and its Euler scheme.  Let $b,\, \sigma:[0,T]\times \R^d\to \R^d$ be two continuous functions, Lipschitz continuous in $x$ uniformly in $t\!\in [0,T]$ so that there exists  real constant $C=C_{b,\sigma}>0$ such that
\begin{equation}  \label{EqEulerLip}
\forall\, t\!\in [0,T], \; \forall\, x\!\in \R^d,\quad \max\big(|b(t,x)|,|\sigma(t,x)|\big) \le C(1+|x|).
\end{equation}

We denote by $[b]_{\rm Lip} $ and $[\sigma]_{\rm Lip}$ the Lipschitz conefficients (in sspace) of $b$ and $\sigma$ respectively. Let $h =\frac Tn>0$, and $t_k= t^n_k = \frac {kT}{n}$, $k\!\in [\![0,n]\!]$. The discrete time Euler scheme reads
\[
\bar X_{t_{k}}= F_{k}(\bar X_{t_{k-1}},Z_k) , \; k=1:n
\]
where $(Z_k)_{k=1:n}$ are i.i.d., ${\cal N}(0;1)$-distributed and
\begin{eqnarray*}
F_{k}(x,Z) & = & x+ b(t_{k-1}, x) h + \sigma(t_{k-1},x)\sqrt{h} Z \\
\,[F]_{\rm Lip}& \le &  \left(\big(1+h[b(t_{k-1},.]_{\rm Lip}\big)^2 + h[\sigma(t_{k-1},.)]^2_{\rm Lip} \right)^{\frac 12}\\
&\le& 1+C_1(T) h
   \end{eqnarray*}
 (with $C_1 =  [b]_{\rm Lip} + \frac 12 [\sigma]^2_{\rm Lip}\, $)   and $F_k(x,Z)$ can be decomposed into $F_k(x,Z)= a+ \sqrt{h}A\,Z$ with
 $a= x+ b(t_{k-1}, x) h$ and $A= \sigma(t_{k-1},.)$
so that
\[
|a|\le |x|(1+Ch)+Ch\quad \mbox{ and }\quad \Vert A \Vert^p= \Vert \sigma(t_{k-1}, x) \Vert^p \le   2^{p-1}C^p(1+|x|^p).
\]

Applying the above Lemma~\ref{lem:key}$(b)$  with $\Upsilon= L=C$,    $\zeta=Z$, yields, that one may set, for every $k\!\in \{1,\ldots,n\}$,
\[
\alpha_{p,k}= \big(  e^{\kappa_p  h } L + K_p \big) h    \quad \mbox{ and }\quad \beta_{p,k}= 1+(\kappa_pe^{\kappa_ph}+K_p)h
\]
where we used that
$e^{ \kappa_p h}  + K_p h \le 1+(\kappa_pe^{\kappa_ph}+K_p)h$ since $e^x\le 1+x e^x$.

\begin{rem}  \label{RemForFioCal} As pointed out in the introduction, this bound include the procedure used in~\cite{CalFioGra2} to quantize the couple price-volatility
 process. In fact, the Euler scheme associated to the volatility process evolves following a Markov chain $X_k^1 = F_k^1(X_{k-1}^1, Z_k)$ whereas the dynamics of Euler scheme associated to the price process is given by   $X_k^2 = F_k^2\big((X_{k-1}^1,X_{k-1}^2), Z_k\big)$, where $(Z_k)$ is a ${\cal N}(0, I_2)$ iid sequence of random variables. As a consequence, setting $X=(X^1,X^2)$, we may write down $X_k = (X_k^1,X_k^2) = F_k(X_{k-1},Z_k)$ where for any $x = (x_1,x_2), z  \in \mathbb R^2$, $F_k(x,z) =  \begin{pmatrix} F_k^1(x_1,z ) \\  F_k^2(x, z) \end{pmatrix}$.
\end{rem}

\subsubsection{Euler scheme of a diffusion with jumps}  \label{SecJump1}
 We start from the following Euler scheme for the jump diffusion \eqref{EqJumpDif}:
\[
\bar X_{k+1}= \bar X_k +h \, b(t_k, \bar X_k) + \sqrt{h}\,\sigma(t_k, \bar X_{t_k})Z_{k+1}+ \gamma(\bar X_k) \big(\tilde \Lambda_{(k+1)h}-\tilde \Lambda_{kh}\big)
\]
where $(\tilde \Lambda_t)_{t\in [0,T]}$ is a compensated Poisson process defined by
\[
\tilde \Lambda_t = \sum_{k=1}^{\Lambda_t}U_k -\lambda\, t \,\E\, U_1 ,
\; t\!\in [0,T],
\]
where $(\Lambda_t)_{t\ge 0}$ is  a standard Poisson process with intensity $\lambda>0$, $(U_k)_{k\ge 1}$  is   i.i.d. sequence of independent    square integrable  random variables,  both are independent and independent
 of the Gaussian white noise $(Z_k)_{k\ge 1}$.

 We assume $\gamma$ is Lipschitz continuous and  $b$, $\sigma$: $[0,T]\times \R\to \R$ are continuous and Lipschitz continuous in $x$, uniformly in $t\!\in [0,T]$. In particular, let $C=C_{b,\sigma, \gamma}>0$ such that
 \[
 \forall\, t\in [0,T],\; \forall\, x\!\in \R,\quad\max\big(|b(t,x)|, |\sigma(t,x)|, |\gamma(x)|\big)\le C(1+|x|).
 \]

Note that, as a classical consequence of Burkholder-Davis-Gundy Inequality, if $U_1\!\in L^p$ for some $p\!\in [1,+\infty)$, every $t\!\in [0,T]$,
$$
\E \, |\tilde \Lambda_{t+s}-\tilde \Lambda_t|^p\le c_p\,(\lambda s)^{\frac p2}\E\, |U_1|^{p}
$$
where $c_p$ is a positive universal constant, only depending on $p$ ($c_2=1$).  Then, one shows that
\begin{eqnarray*}
F_{k}(x,Z) & = & x+ b(t_{k-1}, x) h + \sqrt{h}\Big(\sigma(t_{k-1},x) Z+ \sqrt{\lambda}\|U_1\|_{_2}\gamma(x)\frac{\tilde \Lambda_{h}}{\sqrt{\lambda h}\|U_1\|_{_2}}\Big)  \\
\mbox{with }\quad [F]_{\rm Lip}& \le &  \left(\big(1+h[b(t_{k-1},.]_{\rm Lip}\big)^2 + h\big([\sigma(t_{t_{k-1}}.)]^2_{\rm Lip} + \lambda\,\E\, U_1^2 [\gamma(.)]^2_{\rm Lip} \big)\right)^{\frac 12}\\
&\le& 1+C_1 h.
\end{eqnarray*}
with 
$$
C_1 =   [b]_{\rm Lip} + \frac 12  \big([\sigma(]^2_{\rm Lip} +\lambda \E\, U_1^2 [\gamma]^2_{\rm Lip}\big).
$$ 
Moreover, $F_k(x,Z)$ can also be decomposed into $F_k(x,Z)= a+ \sqrt{h}\,A\,Z$ with
 $a= x+ b(t_{k-1}, x) h$,
 $$
 A=\Big[ \sigma(t_{k-1},.) \quad \gamma(t_k,.)\|U_1\|_{_p}\Big]\quad \mbox{ and }\quad \zeta = \zeta_h = \left[\begin{array}{c} Z\\ \frac{\tilde \Lambda_{h}}{\sqrt{\lambda h} \|U_1\|_{_p}}\end{array}\right]
 $$
 with $d=1$, $q=2$ so that
\begin{eqnarray*}
&&|a|\le |x|(1+Ch)+Ch \\
&\mbox{ and } \quad & \|A\|^p=\Big (\sigma^2(t_{k-1},x)+\lambda \|U_1\|_{_p}^2 \gamma^2(t_{k-1},x)\Big)^{\frac p2} \le   2^{p-1}(1+\lambda  \|U_1\|_{_p}^2)^{\frac p2}C^p(1+|x|^p)
\end{eqnarray*}
and, for every $p\!\in [2,3)$,
\[
\E \,|\zeta|^p = \E\left[Z^2 + \left(\frac{\tilde \Lambda_{h}}{\sqrt{\lambda h} \|U_1\|_{_p}}\right)^2\right]^{\frac p2}\le 2^{\frac p2-1}\left(\E|Z|^p + \E \left|\frac{\tilde \Lambda_{h}}{\sqrt{\lambda h} \|U_1\|_{_p}}\right|^p \right)  \le \tilde c_p = 2^{\frac p2-1} \left(\E |Z|^p + c_p \right).
\]

We may apply  the above Lemma~\ref{lem:key}$(b)$  with $L=C$,  $\Upsilon = (1+\lambda  \|U_1\|^2_{_p})^{\frac 12}C$. Denoting by $\widetilde K_p$ the constant $K_p$ where   $\E\,|\zeta|^p$ is replaced by $\tilde c_p$. This lemma allows us to  set, for every $k\!\in \{1,\ldots,n\}$,
\[
\alpha_{p,k}= \big(  e^{\kappa_p  h } L + \widetilde K_p \big)h    \quad \mbox{ and }\quad \beta_{p,k}= 1+(\kappa_pe^{\kappa_ph}+K_p)h
\]
where we used that
$e^{ \kappa_p h}  + \widetilde K_p h \le 1+(\kappa_pe^{\kappa_ph}+\widetilde K_p)h$ since $e^x\le 1+x e^x$.

\subsubsection{Milstein scheme}
%
%
%
%
%

Assume  $b$, $\sigma$ are ${\cal C}_b^{2}$ [voir dans poly , les conditions exactes] ($i.e.$ $b'_x$ and $\sigma'_x$ bounded and $\sigma\sigma'_x$ Lipschitz continuous in $x$, uniformly in $t\!\in [0,T]$).  We will focus on the one-dimensional Misltein scheme for which we have closed form allowing a fast recursive quantization procedure (see~\cite{McWalter}). 
\[
\bar X_{t_{k}}= F_{k}(\bar X_{t_{k-1}},Z_k) , \; k=1:n
\]
where
\[
F_{k}(x,Z) = x+ b(t_{k-1}, x) h + \sigma(t_{k-1},x)\sqrt{h} Z + \frac h2 \sigma\sigma_x'(t_{k-1},x)(Z^2-1).
\]

Elementary computations show that that $[F_k]_{\rm Lip}$ can be taken as
   \begin{eqnarray*}
[F_k]_{\rm Lip}&\le&  \left(\big(1+h[b(t_{k-1},.)]_{\rm Lip}\big)^2 + h[\sigma(t_{k-1},.)]^2_{\rm Lip} + \frac{h^2}{2}[\sigma\sigma_x'(t_{k-1},.)]^2_{\rm Lip}  \right)^{\frac 12}\\
&\le& 1+C_1(T) h
   \end{eqnarray*}
with $$
C_1  = \sup_{k=0,\ldots,n-1}[b(t_{k-}, .)]_{\rm Lip} + \frac 12 \sup_{k=0,\ldots,n-1}\big([\sigma(t_{k-1}, .)]^2_{\rm Lip} +T \sup_{k=0,\ldots,n-1}\big([\sigma\sigma'_ct_{k-1}, .)]^2_{\rm Lip}\big).
$$ 
since $\E\, (Z^2-1)^2 =  2$.

\bigskip
One still has $a=  x+ b(t_{k-1}, x) h $ but now, with $d=1$ and $q=2$,
$$
A= \displaystyle \Big [\sigma(t_{k-1},.)\;\;  \frac{\sqrt{h}}{2}\sigma\sigma'(t_{k-1},.) \Big]\quad \mbox{ and }\quad \zeta = \displaystyle \left[\begin{array}{c}Z\\Z^2-1\end{array}\right]
$$
so that
\[
|a|\le |x|(1+Ch)+Ch\quad \mbox{ and }\quad \|A\|^p=|\sigma_k(x)|^p\Big(1+\frac h4 (\sigma'_k(x))^2\Big)^{\frac p2}\le 2^{p-1}C^p\Big(1+\frac h4[\sigma]_{\rm Lip}^2\Big)^{\frac p2}.
\]

Set  $L= C$ and $\Upsilon =  \Big(1+\frac T4[\sigma]_{\rm Lip}^2\Big)^{\frac 12}$. Applying Lemma~\ref{lem:key} yields again
\[
\alpha_{p,k}= \big(  e^{\kappa_p  h } L + K_p \big) h    \quad \mbox{ and }\quad \beta_{p,k}= 1+(\kappa_pe^{\kappa_ph}+K_p)h.
\]
As for the implementation of the quantization optimization procedure developed, it is proposed in~\cite{McWalter} to re-write $F_k$ as
\[
F_k(x,z)= x+ b(t_{k-1}, x) h-\frac12\Big( \frac{\sigma}{\sigma'_x}(t_{k-1},x)-h\sigma\sigma'_x(t_{k-1},x))\Big) +\frac{h}{2}\Big(z+\frac{1}{\sigma'_x(t_{k-1},,x)\sqrt{h}}\Big)^2
\]
so that the fast recursive quantization procedure need for that scheme to have analytical formulas  for the c.d.f and the partial first moment functions of {\em uncentered $\chi^2$-distributions} of the form $(Z+c)^2$, $Z\sim{\cal N}(0,1)$, for which closed forms are available. We refer to~\cite{McWalter} for details.

\subsubsection{Simplified weak order 2.0 Taylor scheme}
 This higher order scheme was introduced by~\cite{KloPla} and has been recursively quantized in~\cite{McWalter}. In a $1$-dimensional setting it can be written in an elementary form (without iterated stochastic integrals). To alleviate notations we will assume that the  drift $b$ and the  volatility coefficient $\sigma$ are homogeneous in time $i.e.$ $b(t,x)= b(x)$ and $\sigma(t,x)=\sigma(x)$, both functions being assumed to be twice differentiable. Then it reads, $h=\frac Tn $ still denoting the step of the scheme,
 \[
  \bar X_{t_k}= F_k(\bar X_{t_{k-1}},Z_{k})
\]
where
\[
F(x,z)= x+B_h(x) + C_h(x) z  + D_h(x) (z^2-1)
\]
with
  \begin{align*}
  B_h(x)&= b(x)h +\frac 12 \,\widetilde b(x) h^2\quad\mbox{ with} \quad \widetilde b(x) = bb'(x) +\frac 12 \,b''(x)\sigma^2(x) \\
  C_h(x) &=  \sigma(x)\sqrt{h}  +\frac12\, \widetilde \sigma(x)h^{\frac32} \quad\mbox{ with} \quad \widetilde \sigma(x) =  (b\sigma)'(x)+\frac12 \sigma''(x)\sigma^2(x) \\
  D_h(x) &= \frac 12 \sigma\sigma'(x) h.
  \end{align*}

  In view of the  implementation of the scheme we may mimick  the above square completion with  this formula to make appear again an uncentered $\chi^2$-distribution:
  \[
  f(x,z)= B_h(x) -D_h(x)-\frac{C^2_h(x)}{4D_h(x)}+D_h(x)\left(z-\frac{C_h(x)}{2D_h(x)}\right)^2.
  \]
First we note that under the assumptions made on $b$ and $\sigma$ one easily checks that $b$, $\sigma$, $\widetilde b$, $(\sigma\sigma')^2$ and $\widetilde \sigma$ are all Lipschitz continuous and bounded. As a consequence, one easily checks that
  \[
  [F]_{\rm Lip}\le \left(\Big(1+h[b]_{\rm Lip} +\frac{h^2}{2}[\widetilde b]_{\rm Lip}\Big)^2+ h\Big([\sigma]_{\rm Lip}+\frac{h}{2}[\widetilde \sigma]_{\rm Lip}\Big)^2 + \frac{h^2}{2} [(\sigma\sigma')^2]_{\rm Lip}\right)^{\frac 12}.
  \]

  Then we set similarly to former examples
  \[
  a= x+hb(x) +h^2\,\widetilde b(x)
  \]
  and
 \[
  A= \displaystyle \Big [\sigma(x)+\frac12 \widetilde\sigma(x) h\;\;  \frac{\sqrt{h}}{2}\sigma\sigma'(x) \Big]\quad \mbox{ and }\quad \zeta = \displaystyle \left[\begin{array}{c}Z\\Z^2-1\end{array}\right]
\]
so that, once noted that $0<h\le T$
\[
|a|\le |x| +h\big(\|b\|_{\sup} +T\|\widetilde b\|_{\sup}\big)
\]
and
\[
\|A\|= \left(\Big(\sigma(x)+\frac h2\widetilde \sigma(x)\Big)^2 + \frac{h}{4} (\sigma\sigma')^2(x) \right)^{\frac12}\le C_{T,\sigma,b}(1+h|x|).
\]
Then we may apply Lemma~\ref{lem:key}. $(a)$, with $L= \|b\|_{\sup} +T\|\widetilde b\|_{\sup}$, and $\Upsilon = C_{T, \sigma, b}$,  so that 
\[
\alpha_{p, k} =  \big(  e^{\kappa_p  h } L + K_p \big) h \quad \mbox{ and } \quad \beta_{p,k} =  e^{ \kappa_p h}  + K_p \,h^p.
\]

\section{Weak error rate for recursive quantization}  \label{SectionWeakError}
\subsection{A general weak error rate for smooth functions}
\begin{prop}\label{pro:weakerror}
$(a)$ Let $(\bar X_k)_{k=0:n}$ be an homogeneous Markov chain defined by~\eqref{eq:Markov} with transition kernel $P(x,dy)$. 
Assume that at every instant $k\!\in [\![0,n]\!]$, $\widehat X_k={\rm Proj}_{\Gamma_k}(\widetilde X_k)$  where $\Gamma_k$  is a stationary quantizer. 
Let $\mathcal{V}\subset\mathcal{C}^1(\R^d,\R)$ be a vector subspace satisfying:
\[
\forall\, f\!\in {\cal V},\quad [\nabla f]_{\rm Lip}<+\infty\quad\mbox{ and }\quad P\big(\mathcal{V}\big)\subset \mathcal{V}.
\]
%
%
Then, for every $f\!\in \mathcal{V}$ and every $k\!\in [\![0,n]\!]$,
\[
\big|\E\, f(\widehat X_k)-\E\,f(\bar X_k)\big|\le \frac{1}{2} \sum_{\ell=0}^k[\nabla P^{k-\ell} f]_{\rm Lip}\big\|\widehat X_\ell-\widetilde X_\ell\big\|^2_2.
\]

\noindent $(b)$ If there exists $h>0$ such that $\forall\, f \!\in {\cal V},\; \exists \,C,\,C'>0$ such that  
\begin{equation}\label{eq:majorgrad}
[\nabla P f]_{\rm Lip} \le  e^{Ch} [\nabla f]_{\rm Lip}+C'[f]_{\rm Lip}h,
\end{equation}
 then 
\[
\forall\, k\!\in [\![0,n]\!],\quad \big|\E\, f(\widehat X_k)-\E\,f(\bar X_k)\big|\le \frac{1}{2} \sum_{\ell=0}^k\Big([\nabla f]_{\rm Lip}e^{C(k-\ell)h}+C'[f]_{\rm Lip} t_k\Big)\big\|\widehat X_\ell-\widetilde X_\ell\big\|^2_2.
\]
\end{prop}

\begin{rem} In the non-homogeneous case, one should simply replace  $ [\nabla P^{k-\ell} f]_{\rm Lip}$ by $ [\nabla P_0\cdots P_{k-1} f]_{\rm Lip}$ in Claim~$(a)$. Claim~$(b)$ remains true as set if we assume that for every $k\!\in [\![0,n-1]\!]$, $[\nabla P_k f]_{\rm Lip} \le  e^{Ch} [\nabla f]_{\rm Lip}+C'[f]_{\rm Lip}h$ where $C$  and $C'$ do not depend on $k$.

\end{rem}
\begin{proof}[${\it \textbf{Proof}}$]  As $[\nabla f]_{\rm Lip}<+\infty$, we know that
\[
\big|\E\, f(\widehat X_k)-\E\,f(\bar X_k)\big|\le \big|\E\, f(\widehat X_k)-\E\,f(\widetilde X_k)\big|+\big|\E\, f(\widetilde X_k)-\E\,f(\bar X_k)\big|.
\]
The quantization $\widehat X_k$ of $\widetilde X_k$ being optimal, $\widehat X_k$ is stationary so that
\[
 \big|\E\, f(\widehat X_k)-\E\,f(\widetilde X_k)\big| \le \frac{[\nabla f]_{\rm Lip}}{2}\big\|\widehat X_k-\widetilde X_k\big\|^2_2.
\]
Now, for every $g\!\in\mathcal{V}$, 
\begin{align*}
\big|\E\, g(\widehat X_\ell)-\E\,g(\bar X_\ell)\big|& \le \big|\E\, g(\widehat X_\ell)-\E\,g(\widetilde X_\ell)\big|+\big|\E\, g(\widetilde X_\ell)-\E\,g(\bar X_\ell)\big|\\
								&\le \frac{[\nabla g]_{\rm Lip}}{2} \big\|\widehat X_{\ell}- \widetilde X_{\ell}\big\|_2^2+\big|\E\, g(\widetilde X_\ell)-\E\,g(\bar X_\ell)\big|\\
								&=   \frac{[\nabla g]_{\rm Lip}}{2} \big\|\widehat X_{\ell}- \widetilde X_{\ell}\big\|_2^2+ \big|\E\, Pg(\widehat X_{\ell-1})-\E\,Pg(\bar X_{\ell-1})\big|.
\end{align*}
As $P^\ell f\!\in \mathcal{V}$ for every $\ell\ge 0$ owing to $(i)$,  we derive  by an easy backward induction that
\[
\big|\E\, f(\widehat X_k)-\E\,f(\bar X_k\big|\le \frac{1}{2} \sum_{\ell=0}^k[\nabla P^{k-\ell} f]_{\rm Lip}\big\|\widehat X_\ell-\widetilde X_\ell\big\|^2_2.
\]

\noindent $(b)$ Now it is clear by a forward induction based on~\eqref{eq:majorgrad}  that $[\nabla P^{\ell} f]_{\rm Lip}\le e^{C_f\ell h}$. This completes the proof.
\end{proof}

The key assumption is the stationarity of successive the quantization grids. This is the case when the quanization grids are optimal or when, dealing with product quantization, when a product grid is made  of optimal scalar grids on each marginal suppsoed to be mutually independent.

As far as recursive  product quantization is concerned, this is always the case on a dimension ($d=1$) but turns out to be a rather restrictive condition in higher dimension. It implies in a diffusion fremawork that all the components are independent.
\subsection{Some applications}

\subsubsection{Euler scheme of a Brownian diffusion}  We will consider for the sake of simplicity only the autonomous Euler scheme with step $h=T/n$, still defined by~\eqref{EqEuler} but with $b(t,x) = b(x) $ and $\sigma(t,x)= \sigma (x)$ so that it makes up an $\R^d$-valued  homogeneous Markov chain with transition $P(x, dy)$ defined by
\[
Pf(x) = \E \,f\big({\cal E}_h(x,Z)\big) \quad with \quad \mathcal{E}_h(x,z) = x+h\, b(x) +\sqrt{h}\,\sigma(x)z, \; z\!\in \R^q
\]
and $Z\sim{\cal N}(0,I_q)$. 
\begin{prop}[Euler scheme]\label{prop:Euler}$(a)$ If $b$, and $\sigma$ are twice times differentiable with $Db$, $D^2b$, $D\sigma$ and (all  matrices) $(\partial_{x_{i},x_{j}}^2\sigma) \sigma^*$, $i,j\!\in [\![1,d]\!]$,  are bounded and if $f:\R^d\to \R$ is twice differentiable with a Lipschitz gradient, then there exists a real constant $C= C_{Db,D^2b,D\sigma,(D^2\sigma) \sigma^*}>0$, not depending on $h$, such that   
\[
[Pf]_{\rm Lip} \le (1+Ch) [f]_{\rm Lip}\quad \mbox{ and }\quad [\nabla Pf]_{\rm Lip} \le (1+Ch)[\nabla f]_{\rm Lip} +h\|D^2b\|_{\sup} \|f\|_{\sup}.
\]

\noindent $(b)$ As a consequence, for every $k\in \![\![0,n]\!]$
\[
 [\nabla P^kf]_{\rm Lip}\le e^{Ct_k}\big([\nabla f]_{\rm Lip}+t_k \big)\|D^2b\|_{\sup} \|\nabla f\|_{\sup}.
\] 
\end{prop}

We will detail the proof in the case  $d=q=1$ and $b\equiv 0$ to avoid technicalities, keeping in mind that the computations that follow are close to those used to propagate regulariry when establishing the weak error expansion for the weak error of the Euler scheme of a diffusion.

\begin{proof}[${\it \textbf{Proof}}$] $(a)$ We will extensively use the following  two well-known facts :
\begin{align*}
\mbox{-- }\; [Pf]_{\rm Lip} &\le [f]_{\rm Lip} (1+C_{b,\sigma}h)\quad \mbox{(with $C_{b,\sigma}= [b]_{\rm Lip}Ê+[\sigma]_{\rm Lip}^2/2$)} \\
\mbox{-- }\;  \E \, g'(Z)& = \E\, g(Z)Z\quad \mbox{where $g:\R\to \R$ is differentiable,}\\
\nonumber &\hskip 4cm   \mbox{ $g$, $g'$ with polynomial growth}, \; Z\sim {\cal N}(0,1).
\end{align*}
The first inequality comes from~\eqref{EqEulerLip} and the second, known as Stein's identity, follows from a straightforward integration by parts.  Now, let $f$ is twice differentiable with bounded first two derivatives then
\begin{align*}
(Pf)''(x) &= \E\big[f''({\cal E}_h(x,Z) \big(1+hb'(x)+\sqrt{h}\sigma'(x) Z\big)^2\big]+\sqrt{h}\,\sigma''(x) \E\, \big[f'\big({\cal E}_h(x,Z)\big)Z\big]\\
&\qquad +hb''(x)  \E\, f'\big({\cal E}_h(x,Z) \big)   \\
& =  \E\big[f''({\cal E}_h(x,Z)  \big((1+ hb'(x) + \sqrt{h}\sigma'(x) Z)^2+ h\, \sigma\sigma''(x)\big) \big]+hb''(x)\E\, f'\big({\cal E}_h(x,Z) \big)
\end{align*}
where we used Stein's identity in the second equality  to the function $g(z) = f'\big({\cal E}_h(x,z)\big)$. As a consequence
\begin{align*}
\|(Pf)''\|_{\sup} &\le \| f''\|_{\sup}\sup_{x\in \R} \E\, \big[(1+hb'(x)+ \sqrt{h}\sigma'(x) Z)^2+ h\, |\sigma\sigma''(x)|\big]+ h\| f'\|_{\sup} \|b''\|_{\sup}\\
&\le \| f''\|_{\sup}\big(1+ C'_{b,\sigma}h \big)+ h\| f'\|_{\sup}\|b''\|_{\sup}
\end{align*}
where $C'_{b,\sigma}= 2\|b'\|_{\sup}+\|\sigma'\|^2_{\sup}+\|\sigma\sigma''\|_{\sup}+T \|b'\|^2_{\sup}$. 

Consequently, if we set $C=\max( C_{b,\sigma}, C'_{b,\sigma})$, then $[Pf]_{\rm Lip} \le (1+Ch) [f]_{\rm Lip}$ and 
\[
[(Pf)']_{\rm Lip} \le \|(Pf)''\|_{\sup}\le [f']_{\rm Lip}(1+ C h) + [f]_{\rm Lip}\|b''\|_{\sup}h 
\]
  since one clearly has $ \| f''\|_{\sup}\le  [f']_{\rm Lip}$ and $ \| f'\|_{\sup}\le  [f]_{\rm Lip}$. 

If $f'$ is only Lipschitz continuous, one proceeds by regularization: set $f_{\ve} (x)= \E \, f(x+\ve \zeta)$, $\zeta\sim {\cal N}(0,1)$. Then $f'_{\ve}(x) = \E\, f'(x+\ve \zeta)$ and $f_{\ve}''(x) = \frac{1}{\ve} \E\, [f'(x+\ve \zeta)\zeta]$ so that $\|f'_{\ve}-f'\|_{\sup} \le [f']_{\rm Lip} \ve$ and $[f'_{\ve}]_{\rm Lip} \le [f']_{\rm Lip}$.
Hence, one checks that 
$$
[(Pf_{\ve})']_{\rm Lip}\le  [f'_{\ve}](1+Ch) \le [f']_{\rm Lip}(1+Ch)
$$
and $(Pf_{\ve})' $ converges (unifromly on compact sets) toward $(Pf)'$ which finally implies $[(Pf)']_{\rm Lip}\le  [f']_{\rm Lip}(1+Ch)+ h[f]_{\rm Lip}\|b''\|_{\sup}$.

\smallskip
\noindent $(b)$ First one derives that $[P^kf]_{\rm Lip}Ê\le \|f\|_{\sup}(1+Ch)^k$, $k=0,\ldots,n$, so that 
\begin{align*}
[(P^kf)']_{\rm Lip}&\le  [(P^{k-1}f)']_{\rm Lip}(1+Ch)+ h[P^{k-1}f]_{\rm Lip}\|b''\|_{\sup}\\
&\le [(P^{k-1}f)']_{\rm Lip}(1+Ch)+ h\|b''\|_{\sup} \|f\|_{\rm Lip}(1+Ch)^{k-1}.
\end{align*}
A standard discret time Gronwall argument yields the announced result
\end{proof}
\noindent {\bf Remarks.} $\bullet$ When $d=q=1$ and $\sigma$ is convex, one shows that the above conditions on $\sigma$ can be slightly relaxed by assuming only that $\sigma\sigma'$ is Lipschitz continuous: this follows by an appropriate regularization of $\sigma$ once noted that, in the regular case $h(\sigma'(x))^2 + |\sigma\sigma''(x)|)= h (\sigma\sigma')'(x)$.

\smallskip
\noindent  $\bullet$ Under higher smoothness properties on $b$ and $\sigma$, a similar approach would yield a similar control for $[(Pf)']_{\rm Lip}$ when $P$ is the transition of the Milstein or Taylor 2.0 scheme of an autonomous Brownian diffusion. 

\subsubsection{Milstein scheme of a Brownian diffusion ($d=q=1$)} 
We will still focus on the one-dimensional Misltein scheme for which we have closed form allowing a fast recursive quantization procedure (see~\cite{McWalter}).
Let $h=\frac Tn$, $n\ge 1$. We recall that  the Milstein operator ${\cal M}_h$ of an autonomous Brownian diffusion  is defined in  a $1$-dimensional setting by 
\[
{\cal M}_h(x,z) = x +hb(x) +\sqrt{h}\sigma(x) z +\frac{\sigma\sigma'(x)}{2} (z^2-1), \qquad x,\, z\!\in \R.
\]
We will denote by ${\cal M}'_{h,x}$ and ${\cal M}'_{h,z}$ the partial derivatives of ${\cal M}_{h,x}$ w.r.t. the variables $x$ and $z$ respectively.
We set for convenience $\widetilde \sigma = \sigma \sigma'$ and, as before, we denote by $P$ the transition $Pf(x) = \E f\big({\cal M}_h(x,Z)\big)$, $Z\sim {\cal N}(0,1)$.
\begin{prop}[Milstein scheme]Assume $b= 0$ and $\sigma$ are three time differentiable with  $b'$, $b''$, $\sigma'$, $\sigma''$  and $\sigma\sigma''$ bounded then
\[
[(Pf)']_{\rm Lip}Ê\le (1+Ch)[f']_{\rm Lip}  +h\|b''\|_{\sup} [f]_{\rm Lip}
\]
where $C$ only depends on the sup norms of the above functions and $T$. In particular, the conclusion of item~$(b)$ of the former Proposition~\ref{prop:Euler} still holds.
\end{prop}
\begin{proof} We will only detail the case where the drift $b\equiv0$ to alleviate computations. We will use a second order Stein's identity, namely,  for every twice differentiable function $g:\R\to \R$, 
\[
\E\, g(Z)(Z^2-1) = \E\, g''(Z) = \E\, g'(Z)Z.
\]
We may assume that $f$ is twice differentiable with bounded second derivative. Then, one checks that
\begin{equation}\label{eq:MilsteinTrans}
(Pf)''(x) = \E\left[  f''\big({\cal M}_h(x,Z)\big){\cal M}'_{h,x}(x,Z)^2\right] +  \E\, f'\big({\cal M}_h(x,Z)\big){\cal M}''_{h,x^2}(x,Z).
\end{equation}
Now
\begin{align*}
 \E\, f'\big({\cal M}_h(x,Z)\big){\cal M}''_{h,x^2}(x,Z)&= \sqrt{h}\,\sigma''(x) \E\, f'\big({\cal M}_h(x,Z)\big)Z+ \frac h2 \widetilde\sigma''(x) \E\,[ f'\big({\cal M}_h(x,Z)\big)(Z^2-1)]\\
 &=  \sqrt{h}\,\sigma''(x) \E\,[ f''\big({\cal M}_h(x,Z)\big){\cal M}'_{h,z}(x,Z)] \\
 &\hskip 4cm +\frac h2 \widetilde\sigma''(x)\E\,[ f''\big({\cal M}_h(x,Z)\big) {\cal M}'_{h,z}(x,Z)Z].
\end{align*}

One checks that, for every $x\!\in \R$, 
\begin{align*}
\big|\sqrt{h}\,\sigma''(x) \E\,[ f''\big({\cal M}_h(x,Z)\big){\cal M}'_{h,z}(x,Z)]\big|Ê&\le \sqrt{h}|\sigma''(x) |\|f''\|_{\sup}\| {\cal M}'_{h,z}(x,Z)\|_2 \\
&\le  h  \|f''\|_{\sup}\big(\|\sigma\sigma''\|_{\sup}+\sqrt{h} \|\widetilde\sigma\sigma''\|_{\sup}\big)
\end{align*}
and
\begin{align*}
\big| \widetilde\sigma''(x)\E\,[ f''\big({\cal M}_h(x,Z)\big) {\cal M}'_{h,z}(x,Z)Z]\big|& \le \big| \widetilde\sigma''(x)\big| \|f''\|_{\sup} \| {\cal M}'_{h,z}(x,Z)\|_2\|Z\|_2\\
&\le  \|f''\|_{\sup}\big(\|\widetilde \sigma\sigma''\|_{\sup}+\sqrt{h} \|\widetilde \sigma\widetilde\sigma''\|_{\sup}\big).
\end{align*}
Moreover, using that $\E Z^3= \E\, Z=0$ so that $\E\, Z(Z^2-1) =0$,
\begin{align*}
\E\, {\cal M}'_{h,x}(x,Z)^2&= 1 +h(\sigma'(x) )^2+ \frac{(\widetilde \sigma'(x))^2}{4}h^2\E\, (Z^2-1)^2 \\
&\le 1+ \|\sigma'\|_{\sup} h+\frac{\|\widetilde \sigma'\|^2_{\sup}}{2}h^2.
\end{align*}
Plugging these three inequalities into~\eqref{eq:MilsteinTrans} and keeping in mind that $h$ is always bounded by $T$, we derive from the assumptions made on $\sigma$    the existence of a real constant $C= C(\sigma,T)$  such that 
\[
\| (Pf)''\|_{\sup}\le \|f''\|_{\sup}(1+Ch).
\]
Then, one concludes by regularization  like with the Euler scheme.
\end{proof}

\subsubsection{Euler scheme of a jump model} \label{SecJump2}
We consider the case of an SDE driven by a compound Poisson process with intensity $\lambda>0$ and jump distribution $\mu$
and we denote by $U$ (instaed of $U_1$) a random variable with distribution $\mu$. We will assume that the drift $b$ and the Brownian diffusion coefficient are both zero to enhance the treatment of the jump component. Let $h= T/n$, $n>\lambda T$, $\tilde \lambda = \lambda \E\, U$ and $U_h = U-\tilde \lambda h$.  As $\lambda h\!\in (0,1)$, the Euler scheme~\eqref{EulerJump} is well defined  and its 
 transition is formally defined by 

\begin{equation}\label{eq:JumpTrans}
Pf(x)= (1-\la\,h) f\big(x-\tilde \lambda h\gamma(x)\big) +\la\,h \E\, f\big(x+\gamma(x) U_h\big)
\end{equation}
so that, if $f$ is twice differentiable  

\begin{align*}
(Pf)'(x) &=(1-\la\,h) f'\big(x-\tilde \lambda h\gamma(x)\big)(1-\tilde \lambda h\gamma'(x)) +\la\,h \E\, \big[f'\big(x+\gamma(x) U_h\big)(1+\gamma'(x)U_h) \big]\\
(Pf)''(x) &=(1-\la\,h)\big[ (1-\tilde \lambda h\gamma'(x))^2f''\big(x-\tilde \lambda h\gamma(x)\big)-\lambda h\gamma''(x)  f'\big(x-\tilde \lambda h\gamma(x)\big)\big]\\
&\quad + \la\,h \gamma''(x) \E \,\big[ f'\big(x+\gamma(x) U_h\big)U_h\big]+ \la\,h \E \,\big[ (1+\gamma'(x)U_h)^2 f''\big(x+\gamma(x)U_h\big)\big].
\end{align*}

Assume that  $\mu$ admits a density $p$ so that $\mu(du)= p(u)du$. Then, $U_h \sim p_h(u)du$ with $p_h(u)= p(u+\tilde \lambda h)$. If $\E\, U^2<+\infty$ then $\sup _{0<h\le 1}\E\, U_h^2 <+\infty$ and one easily checks that 
\[
\pi_h(v)= \frac{\int_{-\infty}^{v}up_h(u)du}{\E\, U_h^2}, \; v\!\in \R,
\]
 is a probability density function. Then, by an integration by parts
 \[
 \E\, [f'\big(x+\gamma(x) U_h\big)U_h] = \gamma(x)\, \E\,U_h^2\, \E\, f''\big(x+\gamma(x)V_h\big) \quad \mbox{ with } V\sim \nu = \pi_h(v)dv.
 \]
 Finally, note that $\E \, (1+\gamma'(x)U_h)^2 \le 1+\|\gamma\|_{\sup}\E \, U_h^2$.
Consequently, elementary though tedious computations show that
if $\gamma'$ and $\gamma\gamma''$ are both bounded, then 
\begin{align*}
\|(Pf)''\|_{\sup} &\le \|f''\|_{\sup}\Big[(1-\lambda h)\big((1+\lambda h \|\gamma'\|_{\sup})^2+ \lambda h \|\gamma\gamma''\|_{\sup}\E\, U^2_h\big) + \lambda h \big(1+\|\gamma'\|^2_{\sup}\E \, U_h^2\big) \Big] \\
&\quad  + \|f'\|_{\sup} (1-\lambda h)\lambda h  \|\gamma''\|_{\sup}\\
& \le (1+Ch)  \|f''\|_{\sup}+ \lambda h C' \|f'\|_{\sup}
\end{align*}
where $C$ and $C'$ are two positive real constant depending on $ \|\gamma'\|_{\sup}$, $ \|\gamma\gamma'\|_{\sup}$, $ \|\gamma''\|_{\sup}$, $\E\, U$, $\E\, U^2$ and $\lambda$ (but not on $T$).  Adding a drift component and a Brownian diffusion coefficient leads to the same type of bounds. 

One concludes like for the Euler schemes by regularization. This yields the following proposition.

\begin{prop}[Euler scheme of a jump diffusion]\label{prop:Euler}$(a)$ If $b$, and $\sigma$ are twice times differentiable with $Db$, $D^2b$, $D\sigma$, $D\gamma$ and (all  matrices) $(\partial_{x_{i_0},x_{j_0}}^2\sigma) \sigma^*$ and $(\partial_{x_{i_0},x_{j_0}}^2\gamma) \sigma^*$ are bounded and if $f:\R^d\to \R$ is twice differentiable with a Lipschitz gradient, then there exists a real constant $C= C_{Db,D^2b,D\sigma,D\gamma, (D^2\sigma) \sigma^*, (D^2\gamma) \gamma^*}>0$, not depending on $h$, such that   
\[
[Pf]_{\rm Lip} \le (1+Ch) [f]_{\rm Lip}\quad \mbox{ and }\quad [\nabla Pf]_{\rm Lip} \le (1+Ch)[\nabla f]_{\rm Lip} +h\|D^2b\|_{\sup} \|f\|_{\sup}.
\]

\noindent $(b)$ As a consequence, for every $k\in \![\![0,n]\!]$
\[
 [\nabla P^kf]_{\rm Lip}\le e^{Ct_k}\big([\nabla f]_{\rm Lip}+t_k \big)\|D^2b\|_{\sup} \|\nabla f\|_{\sup}.
\] 
\end{prop}

 \section{Recursive quantization for jump processes}\label{sec:NumerExp}
Recall that a first application of the recursive quantization to pure jump processes  has been in~\cite{CalFioGra3}. There approach requires in particular  a  inverse Fourier transform of the marginal of the underlying process  and applied  to jump processes  with  explicit or efficiently computable  characteristic function.  The approach we present here is more general is based on the Euler scheme associated to considered jump diffusion. 

We will temporarily consider slightly more general time discretization schemes than those analyzed in Section \ref{SecJump1} and in Section \ref{SecJump1}  by taking into account the opportunity of more than a single jump during one time step, possibly allowing for coarser discretization. 

\vskip 0.3cm

\noindent $\leadsto$ {\it The algorithm}.  The recursive quantization algorithm of the Euler scheme associated to the jump diffusion  \eqref{EqJumpDif} reads as  \eqref{EqAlgorithmIntro}  where the Euler operator ${\cal E}_k$ is written as a function of the increments  $\Delta \Lambda$ of the Poisson process and the sizes $U_{\ell'}$ of the jumps up to $\Delta \Lambda$. When $\Delta \Lambda = m$ and $U_{\ell'} = u_{\ell'}$ we have 
\[
{\cal E}_k\big(x, z, m, (u_{1}, \ldots, u_m)\big) =  x + h\,  b(t_k,x)  + \sqrt{h} \,\sigma(t_k,x)  z + \gamma(x) \Big(\sum_{\ell'=1}^{m}  u_{\ell'}  - \lambda h\, \E\, U_1 \Big).
\]

\medskip 

\noindent $\leadsto$ {\it The distortion function}.  Recall that the distortion function  $\bar D_{k}$ associated to $\bar  X_{k}$  is given    for every $k=0, \ldots,n-1$ by
   \[
   \bar  D_{k+1}(\Gamma_{k+1}) =  \mathbb E \Bigg[   {\rm dist} \bigg({\cal E}_k \Big(\bar X_k,Z_{k+1},\sum_{\ell=1}^{\Delta \Lambda_{k+1}} U_{\ell}\Big), \Gamma_{k+1}\bigg)^2 \Bigg].  \label {EqDistorNonQuant}
  \]
 Suppose that  $\bar  X_k$ has already been quantized  by  $\hat  X_k^{\Gamma_k}$ and let us set  
 \[
 \widetilde X_{k+1}  =  {\cal E}_k \Big(\hat  X_k^{\Gamma_k} ,Z_{k+1},   \sum_{\ell=1}^{ \Delta \Lambda_{k+1}} U_{\ell}\ \Big).
 \]
   One may approximate the distortion function $\bar D_{k+1}(\Gamma_{k+1})$ by the  (recursive)-distortion  function  $\tilde D_{k+1}(\Gamma_{k+1})$ defined as
     \begin{eqnarray}
  \tilde D_{k+1} (\Gamma_{k+1}) & := &   \mathbb E \big[   {\rm dist}(\widetilde X_{k+1}, \Gamma_{k+1})^2 \big] \nonumber \\
  & = & \mathbb E \Big[   {\rm dist}\Big({\cal E}_k \Big(\hat  X_k^{\Gamma_k} ,Z_{k+1},  \sum_{\ell=1}^{ \Delta \Lambda_{k+1}} U_{\ell}\ \Big), \Gamma_{k+1} \Big)^2 \Big] \nonumber \\ 
  &  = &  \sum_{i=1}^{N_{k}}  \mathbb E \Big[   {\rm dist} \Big({\cal E}_k\Big(x_k^{i},Z_{k+1},  \sum_{\ell=1}^{ \Delta \Lambda_{k+1}} U_{\ell}\Big), \Gamma_{k+1} \Big)^2 \Big] p_k^i  \label{EqDistortion}
 \end{eqnarray}
 where \, $p_k^i = \mathbb P\big(\hat  X_k^{\Gamma_k} =x_k^{i} \big)$.  

 \medskip 
 
\noindent $\leadsto$ {\it How to compute the recursive quantizers}. Stating from \eqref{EqDistortion}, the see that the  recursive-distortion  function associated to the marginal r.v. $\widetilde X_{k+1}$ reads 
\begin{eqnarray}
 \tilde D_{k+1} (\Gamma_{k+1})   &  = &  \sum_{i=1}^{N_{k}}  \mathbb E \Big[   {\rm dist} \big({\cal E}_k(x_k^{i},Z_{k+1},   \sum_{\ell=1}^{ \Delta \Lambda_{k+1}} U_{\ell} \Big), \Gamma_{k+1} \big)^2 \Big] p_k^i  \nonumber \\
&=&  \sum_{i=1}^{N_{k}}  \sum_{m=0}^{+\infty} p_k^i\, p_m  \int_{\mathbb R^m} \prod_{\ell=1}^m \nu(d u_{\ell})  \mathbb E \Big[  {\rm dist} \Big({\cal E}_k \big(x_k^{i},Z_{k+1},  \sum_{\ell=1}^m u_{\ell} \big), \Gamma_{k+1} \Big)^2  \Big]
\end{eqnarray}
 where $p_m = \PP ( \Delta \Lambda_{k+1} = m) =e^{-\lambda h}  \frac{(\lambda h)^{m}}{m!}$.  Our aim is then to compute the sequence $(\Gamma_k)_{0 \le k \le n}$ of optimal quantizers  defined for every $k \in \{1, \ldots, n\}$ by 
  \begin{equation} \label{EqOptQuant}
\Gamma_k \!\in    \arg\min \big\{ \tilde D_{k}(x),\, x\!\in  (\mathbb R^d)^{N_k}\big \},
\end{equation}
 supposing that $\bar X_0$ has already been quantized as $\Gamma_0$.  We discuss with respect to two main situations: when the jump sizes are normally  distributed and for a given general distribution $\nu$. We remark however that the recursive-distortion may be simplified in the short time situation (when $h \approx 0$), making the computations more easy. Since the sort time  situation  is the usual  framework, we will consider that  framework from now on.

\subsection{The short time framework} 
 
It is the  situation where  $h \approx 0$ and where we consider that there is a most one jump during a time step. In this case,  we may consider that for every $k =1, \ldots, n$,   $ \Delta \Lambda_{k}$  has a Bernoulli distribution with 
 \[
 \mathbb P(\Delta \Lambda_k = 1) = \lambda h  \qquad \textrm{ and } \qquad  \mathbb P(\Delta \Lambda_k =0) = 1 - \lambda h. 
 \]
In this case, $ \tilde D_{k+1}$ reads
 \begin{eqnarray*}
 \tilde D_{k+1} (\Gamma_{k+1})   &=&  (1-\lambda h)  \sum_{i=1}^{N_{k}} p_k^i\, \mathbb E \Big[  {\rm dist} \Big({\cal E}_k \big(x_k^{i},Z_{k+1},  0 \big), \Gamma_{k+1} \Big)^2  \Big] \\
 &+& \lambda h  \sum_{i=1}^{N_{k}} p_k^i\,  \int_{\mathbb R} \nu(d u)  \mathbb E \Big[  {\rm dist} \Big({\cal E}_k \big(x_k^{i},Z_{k+1},  u \big), \Gamma_{k+1} \Big)^2  \Big].
\end{eqnarray*}

We next consider  the case where the distribution $\nu$ of $U_1$ is  a gaussian distribution with mean $\mu$ and variance $\vartheta^2$ before considering the general case.  

\medskip 


\noindent $\diamond$ {\it When the jump size has a Normal distribution}.  We suppose that for every $\ell \ge 1$, $U_{\ell} \sim {\cal N}(\mu,\vartheta^2)$.  Then, when  $ \Delta \Lambda_{k+1} =m \in \{0, 1\}$,  we have
 \begin{eqnarray*}
(\bar X_{k+1} \vert  \bar X_{t_k}=x ) &\stackrel{\cal L}{=}& x + h\,  b(t_k,x) + \mu\, (m - \lambda h ) \gamma(x)  + \sqrt{ h \sigma^2(t_k,x) + m \vartheta^2  \gamma^2(x) }\,  \xi_{k+1} 
\end{eqnarray*}
where $(\xi_k)_{k=1,\ldots, n}$ is an   i.i.d., sequence of ${\cal N}(0;1)$-distributed random variables, independent from  $\bar X_0$.  In this case, the distortion reads 
\begin{equation}  \label{EqDistGauss}
 \tilde D_{k+1} (\Gamma_{k+1}) =   \sum_{i=1}^{N_{k}}  \sum_{m=0}^{1}  \mathbb E \Big[  {\rm dist} \Big({\cal E}^g_k \big(x_k^{i},m, \xi_{k+1} \big), \Gamma_{k+1} \Big)^2  \Big]   p_k^i  \, p_m  
\end{equation}
where $p_0  = 1- \lambda h$, $p_1 = \lambda h$, and where   for every $x\in \mathbb R^d , z \in \R$, $m \in \{0, 1\}$,
\begin{equation}
{\cal E}^g_k (x,m,z) =  x + h\,  b(t_k,x) +  \mu\, (m - \lambda h ) \gamma(x)  + \sqrt{ h \sigma^2(x) +  m \vartheta^2  \gamma^2(x) }\, z.
\end{equation}
Set $\mu_k(m,x) =  x + h\,  b(t_k,x) +\mu\,  (m - \lambda h ) \gamma(x) $ and $v_k(m,x) = \sqrt{ h \sigma^2(t_k,x) +  m\vartheta^2  \gamma^2(x) }$.  We also set for every $k=0,\ldots,n-1$ and every    $j=1,\ldots,N_{k+1}$,
$$ 
x_{k+1}^{j-1/2} = \frac{ x_{k+1}^{j} + x_{k+1}^{j-1}  }{2}, \  x_{k+1}^{j+1/2} = \frac{ x_{k+1}^{j} + x_{k+1}^{j+1}  }{2}, \ \textrm{ with }  x_{k+1}^{1/2} = -\infty,  x_{k+1}^{N_{k+1}+1/2} =+\infty, 
$$ 
and  define
$$ 
x_{k+1}^{j-}(m,x): = \frac{x_{k+1}^{j-1/2} - \mu_{k}(m,x) }{v_k(m,x)} \quad \textrm{ and } \quad x_{k+1}^{j+}(m,x): = \frac{x_{k+1}^{j+1/2} - \mu_{k}(m,x) }{v_k(m,x)},\; k=0,\ldots,n-1.
$$

  We may compute the components of the gradient vector and the Hessian matrix associated with this distortion function using  standard computations similar to~\cite{PagSag}  (see Appendix \ref{AppenB}).   

\begin{prop} 
  The transition probability $p_k^{ij}= \mathbb P\big(\widetilde X_{k+1}\!\in C_j(\Gamma_{k+1})\vert \widetilde X_{k} \!\in C_i(\Gamma_{k})\big)$ is given by
  \setlength\arraycolsep{3pt}
\begin{eqnarray}
 p_k^{ij}& =&  \sum_{m=0}^{1}   p_m \Big( \Phi_0(x_{k+1}^{j+}(m,x_{k}^i))  -   \Phi_0(x_{k+1}^{j-}(m,x_{k}^i))\Big)  .
  \label{EqEstProba}
\end{eqnarray}

 \noindent  The probability $p_{k+1}^j= \mathbb P\big(\widetilde X_{k+1}\!\in C_j(\Gamma_{k+1})\big)$  is given for every $j=1,\cdots,N_{k+1}$ by 
\setlength\arraycolsep{3pt}
\begin{eqnarray}
p_{k+1}^j & =&  \sum_{i=1}^{N_k} \sum_{m=0}^{1} p_k^i\, p_m  \Big(   \Phi_0(x_{k+1}^{j+}(m,x_{k}^i))  - \Phi_0(x_{k+1}^{j-}(m,x_{k}^i))  \Big) .
  \label{EqEstProba}
\end{eqnarray}
In the previous expressions, $p_0  = 1- \lambda h$, $p_1 = \lambda h$.

\end{prop}

%
We next   write down the modified distortion function   when the distribution of the jump size is  $\nu$.

\medskip 

\noindent $\diamond$ {\it When the jump size has  a given distribution $\nu$}.   We suppose here  that for every $\ell \ge 1$, $U_{\ell} \sim \nu$. In this case 
\begin{equation*}
 \tilde D_{k+1} (\Gamma_{k+1})   =   \sum_{i=1}^{N_{k}}  \sum_{m=0}^{1} p_k^i\, p_m  \int_{\mathbb R}  \nu(d u)  \mathbb E \Big[  {\rm dist} \Big({\cal E}_k \big(x_k^{i},Z_{k+1},  u \, m  \big), \Gamma_{k+1} \Big)^2  \Big]
\end{equation*}
 where  $p_0  = 1- \lambda h$ and  $p_1 = \lambda h$.   The components of the gradient vector and the Hessian matrix   of the distortion function may be computed using standard computations similar to~\cite{PagSag} (see Appendix \ref{AppenC}).  We may also compute the weights and transition weights via semi-closed formulae.

\subsection{Numerical experiment: pricing of put option in a  Merton jump model}
We consider in the section a european  put option pricing problem where the underlying asset  $X$ evolves (under a risk neutral probability) following  the dynamics: 
\begin{equation} \label{EqBSmodel}
dX_t  = r X_t dt + \sigma X_t dW_t + X_{t-} d \tilde{\Lambda}_t , \qquad X_0=x_0 
\end{equation}
where  $W$ is a Brownian motion,  $\tilde{\Lambda}$ is the compensated compound Poisson process defined by $ \tilde{\Lambda}_t= \sum_{i=1}^{\Lambda_t}  U_i - \lambda  \mathbb E( U_1) \, t$, where the $U_i$ are i.i.d. random variable defined by $ U_i = e^{ \xi_i}-1$ with   $\xi_i = {\cal N}(0; \vartheta^2)$ and $\Lambda$ is a Poisson process with intensity $\lambda>0$, independent (with all the $U_i$'s) from $W$.  The solutions of \eqref{EqBSmodel} reads for every $t \in [0,T]$,
\[
X_t  = X_0\, \exp \Big((r- \frac{\sigma^2}{2}) t  + \sigma W_t \Big) \prod_{i=1}^{\Lambda_t}  e^{\xi_i} .
\]

\medskip 

Denote by $\Phi_0(\cdot)$, the cdf of the ${\cal N}(0,1)$ and  by 
\begin{eqnarray*}
&&  P_{_{\rm BS}}(x,\sigma, r, \tau) = - x\, \Phi_0 (- d_{1}(x, \sigma, r, \tau)) +  e^{-r \tau} K\, \Phi_0(-d_{2}(x, \sigma, r, \tau)) \\
&\textrm{ with } & d_1(x, \sigma, r, \tau) = \frac{1}{\sigma \sqrt{\tau}} \Big( \log \frac{x}{ K} + (r+ \frac{\sigma^2}{2}) \tau\Big)  \quad \textrm{ and } \quad d_{2}(x, \sigma, r, \tau)  = d_1(x, \sigma, r, \tau) - \sigma \sqrt{\tau},
\end{eqnarray*}
the price of the standard Black-Scholes-Merton put price on a geometric Brownian motion with volatility $\sigma$ when the interest rate is $r$, the current stock price is $x$, the time to maturity is $\tau$, and the strike price is $K$.  Then, the risk neutral price $P_0$ at time $t=0$ of the put which underlying asset evolves following \eqref{EqBSmodel} is given by 
\begin{eqnarray}
P_0 &=& e^{-r T}  \mathbb E \big( \max(K -X_T ,0) \big) \nonumber\\
&=& e^{-rT} \sum_{k=0}^{+\infty}  \frac{(\lambda T)^k}{k!} P_{_{\rm BS}} \Big(x_0 e^{\frac{k\vartheta^2}{2} - \lambda T \mathbb E U_1}, \Big(\sigma^2 + \frac{k \vartheta^2}{T}\Big)^{^{1/2}}, r, T \Big).  \label{EqPriceBSJ}
\end{eqnarray}

Our aim is now to compare the call prices we obtain using the recursive quantization with the true price given by \eqref{EqPriceBSJ}.  Using the recursive quantization, the price $P_0$ is approximated by 
\begin{equation}  \label{EqHatC_0}
\widehat P_0  =  e^{-r T}  \mathbb E \big( \max(K - \widehat X_{t_n} ,0) \big)  = e^{-rT}\sum_{i=1}^{N_n}    \max(K- x_n^{i},0)\, \mathbb P(\widehat X=x_n^i)
\end{equation}
for  a regular time discretization steps $t_k = \frac{k T}{n}$ on the interval $[0,T]$ and where  $\widehat X_{t_n}$ is the (optimal) recursive quantization (on the grid $\Gamma_n=\{x_n^1, \ldots,x_n^{N_n}\}$ of size $N_n$) of the marginal random variable $\widehat X_{t_n}$ induced by the Euler scheme associated with \eqref{EqBSmodel}.

\begin{figure}[htpb]
 \begin{center}
  \!\includegraphics[width=12.0cm,height=15cm,angle=-90]{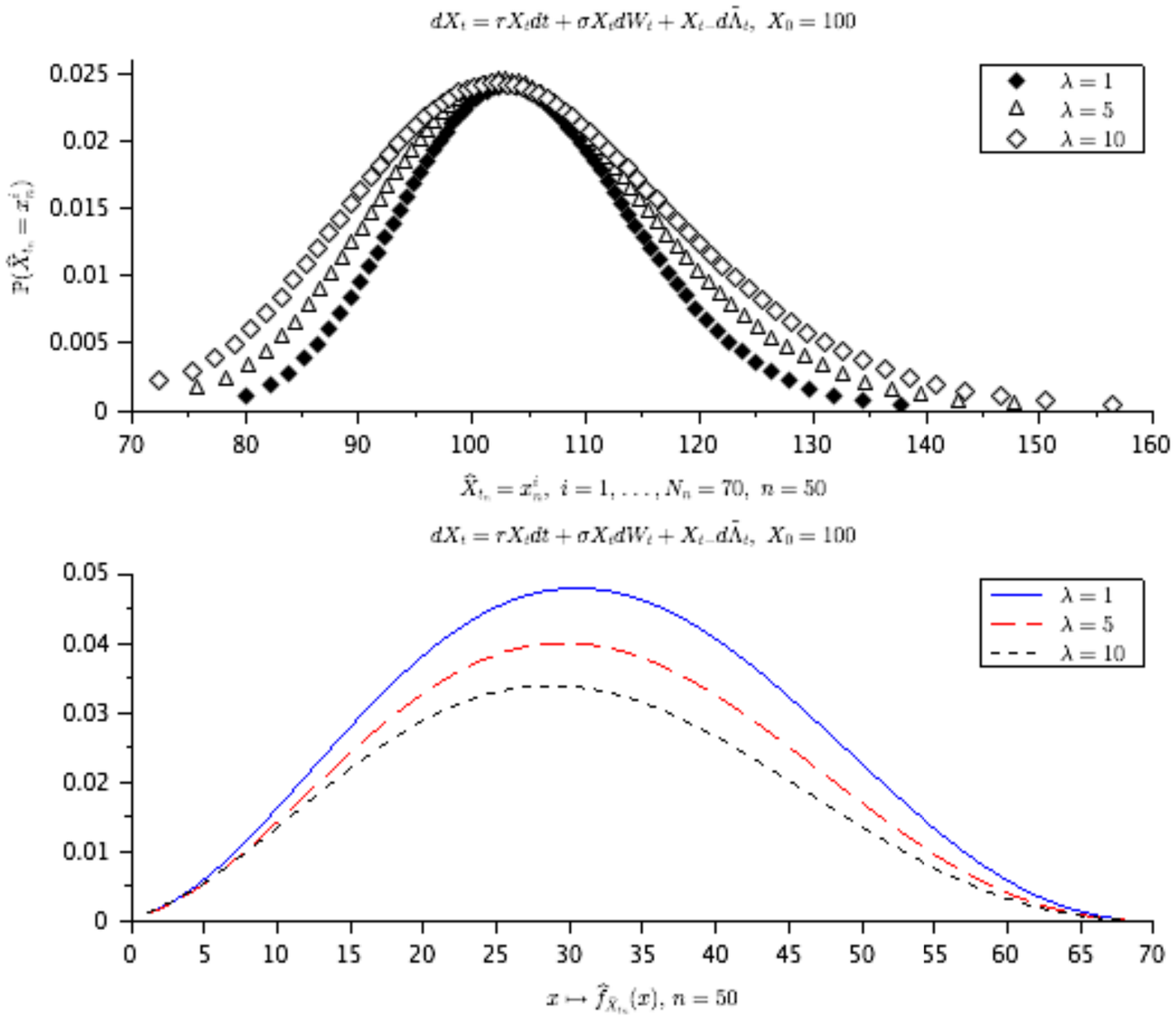}	
  \caption{\footnotesize  (impact of $\lambda$) The model is $d X_t = r\, X_t dt + \sigma X_t dW_t + X_{t-} d \tilde{\Lambda}_t$, $X_0=100$, $\tilde{\Lambda}_t= \sum_{i=1}^{\Lambda_t}  U_i - \lambda t \mathbb E  U_1$, where $ U_i = e^{ \xi_i}-1$ with   $\xi_i = {\cal N}(0; \vartheta^2)$ and $\Lambda$ is a Poisson process with intensity $\lambda$. We choose   $r=0.08$, $\sigma=0.108$,  $\vartheta=0.04$, $T=0.5$. For the quantization, the number of discretization step $n=50$ and $N_k=70$, $\forall \, k =1, \ldots,n$.   We compare de distributions of $\widehat X_{t_n}$ and the densities estimate functions $x \mapsto \hat f_{ \bar X_{t_n}}(x) = 2 \,\mathbb P(\widehat X_{t_n} = x_n^i) /(x_n^{i+1} - x_{n}^{i-1}) \mathds{1}_{[x_n^{i-1},\ x_n^{i}]} (x)$,  $\, x \in [x_n^{2},  x_n^{N_n-1}]$, $n=50$, for $\lambda \in \{1, 5, 10 \}$.} 
  \label{figure1}
  \end{center}
\end{figure}

\noindent $\diamond$ {\em Impact of  $\lambda$ and $\vartheta$  on the marginal distributions of the stochastic process \eqref{EqBSmodel}}.  Before dealing with the numerical experiments on the pricing, we want to see how  the recursive quantization of the Euler scheme $(\bar X_{t_k})_{k=0, \ldots,n}$ looks like. To this end and to see the impact of the intensity $\lambda$ of the jumps on the marginal distributions of the stochastic process \eqref{EqBSmodel}, we compare in Figure \ref{figure1},  the distributions of the recursive quantization $\widehat X_{t_n}$ and   the associated (truncated) marginal densities approximate functions  $\hat f_{\bar X_{t_{50}}}$, for the values of $\lambda \in \{1, 5, 10\}$.  The truncated densities approximate function  $\hat f_{\bar X_{t_k}}$  is defined on  $[x_k^{2},  x_k^{N_k-1}]$ as 
$$\hat f_{X_k}(x) =  \frac{2 \,\mathbb P(\widehat X_k = x_k^i) }{(x_k^{i+1} - x_{k}^{i-1}) }\,\mathds{1}_{[x_k^{i-1},\ x_k^{i}]} (x), \qquad x \in [x_k^{2},  x_k^{N_k-1}].$$
For the numerical tests we use the following set of parameters: $X_0=100$, $n=80$, $T=0.5$, $N_k =70$ for every $k=1, \ldots, n$ and $\widehat X_0 = \mathds{1}_{\{x_0\}}$. We also set  $r=0.08$, $\sigma = 0.108$, $\vartheta=0.04$.   The plots of Figure \ref{figure1}  show that the higher $\lambda$ is, the larger will be the tails (which are not represented in these plots) of the distributions of the marginals of the stochastic process \eqref{EqBSmodel}. 

\begin{figure}[htpb]
 \begin{center}
  \!\includegraphics[width=12.0cm,height=15cm,angle=-90]{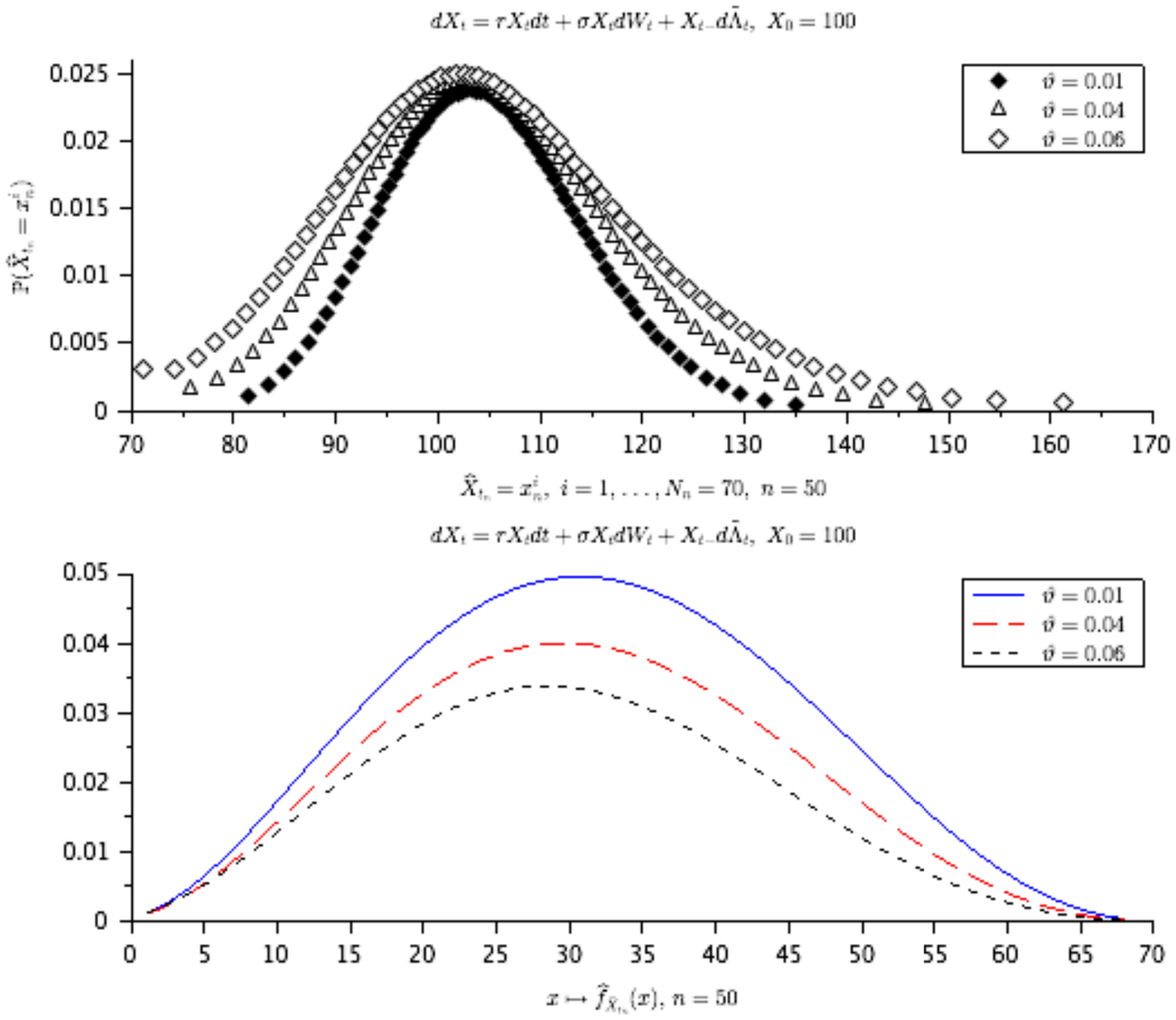}	
  \caption{\footnotesize  (Impact of $\vartheta$) The model is $d X_t = r\, X_t dt + \sigma X_t dW_t + X_{t-} d \tilde{\Lambda}_t$, $X_0=100$, $\tilde{\Lambda}_t= \sum_{i=1}^{\Lambda_t}  U_i - \lambda t \mathbb E  U_1$, where $ U_i = e^{ \xi_i}-1$ with   $\xi_i = {\cal N}(0; \vartheta^2)$ and $\Lambda$ is a Poisson process with intensity $\lambda$. We choose   $r=0.08$, $\sigma=0.108$,  $\lambda=5$, $T=0.5$. For the quantization, the number of discretization step $n=50$ and $N_k=70$, $\forall \, k =1, \ldots,n$.   We compare de distributions of $\widehat X_{t_n}$ and the densities estimate functions $x \mapsto \hat f_{ \bar X_{t_n}}(x) = 2 \,\mathbb P(\widehat X_{t_n} = x_n^i) /(x_n^{i+1} - x_{n}^{i-1}) \mathds{1}_{[x_n^{i-1},\ x_n^{i}]} (x)$,  $\, x \in [x_n^{2},  x_n^{N_n-1}]$, $n=50$, for $\vartheta \in \{0.01, 0.04, 0.06 \}$.} 
  \label{figure2}
  \end{center}
\end{figure}

In Figure \ref{figure2}, we compare the same functions as in Figure \ref{figure1} but this time by putting $\lambda=5$ and by making varying $\vartheta \in \{0.01, 0.04, 0.06\}$ to see the influence of the parameter $\vartheta$ on the marginal distributions of the stochastic process \eqref{EqBSmodel}.  As expected, we see  once again that the higher $\vartheta$ is, the  larger will be  the tails of the distributions of the marginals of the stochastic process \eqref{EqBSmodel}. 

Remark that, compared with the model without jump, there are additional integral terms with respect to the distribution $\nu$ of the jump size when computing the gradient and the Hessian matrix of the distortion function. These integrals are approximated using the optimal quantization method of size $N=50$. This may  increase a little bit  the computation time w.r.t. to  models without jump. In our example, we get all the marginal distributions with their associated weights in around 1min and 30s, using {\it scilab} software in a CPU 3.1 GHz and 16 Gb memory computer. Our aim here is just to test the performance of our method,  not to optimize the execution time. However, it is clear that making the code in C program (even, optimizing the {\it scilab} code) will reduce drastically the computation time since their is many {\it for} loops in the actual code that increase the computation time.  

\medskip

\noindent $\diamond$ {\em Pricing of a European put option with jump process using the recursive quantization}.  Let us come back  to the pricing problem where our aim is  to test the performance of the  recursive quantization method. To this end, we compare the put price $\widehat P_0$ obtained from the recursive quantization method using the formula \eqref{EqHatC_0} with the  true call price which formula is given from Equation \eqref{EqPriceBSJ}. The comparison is done with the following set of parameters:  $r=0.08$, $\sigma=0.07$, $T=0.5$,  the number of discretization step $n=50$ and the size of the quantizations  $N_k=100$, $\forall \, k =1, \ldots,n$, with $N_0=1$. We make varying $\lambda$ in the set  values $\{1, 3, 5 \}$, $\vartheta$ in the set $\{0.01, 0.04\}$ and the strike $K$ in the set values $\{90, 92, 94, 96, 98, 100  \}$ and display the results in Table \ref{table1}. 
 
 For matters of comparaison with the Black-Scholes model where the underlying asset price evolves following the dynamics $dX_t = r X_t dt + \sigma_{_{\rm BS}} dW_t$, a computation of  ${\rm Var} (\ln (X_t))$ in both models allows us to write down $\sigma_{_{\rm BS}}$ (the equivalent volatility in the Black-Scholes model) with respect to $\sigma$:   $\sigma_{_{\rm BS}} =  \sqrt{\sigma^2 + \lambda \vartheta^2}$.

 The numerical results show a maximal absolute error of order  $10^{-2}$ for a Black-Scholes  equivalent volatility $\sigma_{_{\rm BS}} = 0.1135782$ (obtained when  $\lambda=5$ and $\vartheta=0.04$) and a minimal absolute error of order  $10^{-3}$ for a Black-Scholes  equivalent volatility $\sigma_{_{\rm BS}} = 0.0707107$ (obtained with $\lambda=1$ and $\vartheta=0.01$). We also depict  in Table \ref{table1} the true price $P_{_{\rm BS}}$ (and the approximate price $\widehat P_{_{\rm BS}}$  from the recursive quantization, see the numerical examples in~\cite{PagSagMQ} for more detail) of the put in the Black-Scholes model in order to compare it with the true price $P_0$ in the jump model \eqref{EqBSmodel}. We see, as expected, that these two prices tend to coincide when $\lambda$ and $\vartheta$ are small.

   \bigskip

  \begin{table}[h!]
\begin{center}
{\small
\begin{tabular}{|c|c|c|c|c|c|}
\hline
Strike  &$\lambda$& $P_0 / \widehat P_0$ ($\vartheta=1\%$)&  $P_{_{\rm BS}}  / \widehat P_{_{\rm BS}} (\vartheta=1\%)$  &$P_0 / \widehat P_0$ ($\vartheta=4\%$) & $ P_{_{\rm BS}} / \widehat P_{_{\rm BS}}  (\vartheta=4\%)$  \\
\hline
 & 1 &  0.002 / 0.002 &  0.002 / 0.002& 0.015 / 0.013 & 0.009  / 0.008 \\
90  & 3 & 0.003 / 0.002& 0.003 / 0.002& 0.057 / 0.055 & 0.043 / 0.041  \\
 & 5 &0.004 / 0.003 & 0.004 / 0.003&0.120 / 0.118 &  0.104 / 0.101 \\
\hline
 & 1&0.010 / 0.009& 0.010 / 0.009 &0.037  / 0.036 &0.029 / 0.027\\
   92  & 3 &0.012 / 0.011& 0.012  / 0.115& 0.115 / 0.114 &0.100  / 0.098\\
 & 5 & 0.014 /  0.013  &0.014 / 0.013& 0.214 / 0.213& 0.203 /  0.200\\
 \hline
 & 1 &0.035 /  0.035 &  0.036  / 0.035 & 0.087 / 0.087&0.080 / 0.078\\
    94  & 3 &0.041 / 0.040& 0.040 / 0.040&0.218 / 0.218 &0.211 / 0.207\\
 & 5 &0.045 / 0.045& 0.047 / 0.045 &0.368 / 0.370 & 0.371 / 0.367  \\
\hline
 & 1 &0.103 / 0.105 &  0.107  / 0.105 &0.193 / 0.196 &0.193 / 0.190\\
    96  & 3 & 0.117 / 0.113 & 0.116 / 0.115&0.396 / 0.405 &0.406 / 0.402\\
 & 5 &0.124 / 0.129  & 0.126  / 0.126& 0.607 /  0.617&0.635 / 0.629 \\
\hline
 & 1 &0.259 / 0.270 &  0.270  / 0.267 & 0.356 / 0.407&0.414 / 0.410\\
    98  & 3 &0.289 / 0.280& 0.290  /  0.286 & 0.684 / 0.704&0.724 / 0.719\\
 & 5 & 0.296 / 0.300 & 0.310 / 0.304 &0.961 /   0.982&  1.022 / 1.016 \\
\hline
 & 1 &0.566 / 0.585 &  0.589  / 0.585 & 0.751 / 0.775 &0.796 / 0.791\\
    100  & 3 &0.593 / 0.610 & 0.617  / 0.613& 1.121 / 1.159 &1.200 / 1.194\\
 & 5 & 0.620 / 0.640 & 0.650 / 0.641 & 1.459 /  1.499 & 1.561  / 1.554 \\
\hline
\end{tabular}
}
\end{center}
\caption{ \footnotesize The model is $d X_t = r\, X_t dt + \sigma X_t dW_t + X_{t-} d \tilde{\Lambda}_t$, $X_0=100$, $\tilde{\Lambda}_t= \sum_{i=1}^{\Lambda_t}  U_i - \lambda t \mathbb E  U_1$, where $ U_i = e^{ \xi_i}-1$ with   $\xi_i = {\cal N}(0; \vartheta^2)$ and $\Lambda$ is a Poisson process with intensity $\lambda$. We choose   $r=0.08$, $\sigma=0.07$, $T=0.5$. For the quantization, the number of discretization step $n=50$ and $N_k=100$, $\forall \, k =1, \ldots,n$. $P_0$ (resp. $P_{_{\rm BS}}$) is the true put price in the Merton model with jump (resp. without jump) and $\widehat P_0$ and $\widehat{P}_{_{\rm BS}}$ are their respective recursive quantization approximation.} \label{table1}
\end{table}

 \newpage
  
\bibliographystyle{plain}
\bibliography{NLfilteringbib}

\bigskip 

\begin{appendix}

\section{Proof of Lemma \ref{lem:key}}

\begin{proof}[\textit{Proof (of the key Lemma \ref{lem:key})}] $(a)$ It follows from the elementary inequality
\begin{equation*}   \label{LemPropPrinc}
\forall u\!\in \mathbb R^d,\quad \vert  a + u \vert^p  \le \vert a \vert^p + p \vert a \vert^{p-2} (a \vert u) + \frac{p(p-1)}{2} \big(\vert a \vert^{p-2} \vert u \vert^2 + \vert u \vert^p  \big) 
\end{equation*}
that  
\[
\vert a + \sqrt{h} A \zeta \vert^p  \le \vert a \vert^p + p  h^{\frac{1}{2}} \vert a \vert^{p-2}  (a \vert A \zeta) + \frac{p(p-1)}{2}\big( \vert a \vert^{p-2}  h  \vert  A \zeta \vert^2  + h^{\frac{p}{2}}  \vert  A \zeta \vert^p  \big). 
\]
 Applying Young's inequality with conjugate  exponents $p' = \frac{p}{p-2}$ and $q' = \frac{p}{2}$, we get 
 \[
  \vert a \vert^{p-2}  h  \vert  A \zeta \vert^2  \le h \Big(  \frac{\vert a \vert^p}{p'} + \frac{\vert  A \zeta \vert^p}{q'} \Big),
 \]
 which leads to 
  \begin{eqnarray*}
 \vert a + \sqrt{h} A \zeta \vert^p &  \le &  \vert a \vert^p + p  h^{\frac{1}{2}} \vert a \vert^{p-2}  (a \vert A \zeta) + \frac{p(p-1)}{2}\Big(  \frac{h}{p'}  \vert a \vert^{p} + \Big(\frac{h}{q'} + h^{\frac{p}{2}} \Big)   \vert  A \zeta \vert^p   \Big)  \\
 &  \le & \vert a \vert^p \Big( 1 +  \frac{p(p-1)}{2 p'} h \Big) + p  h^{\frac{1}{2}} \vert a \vert^{p-2}  (a \vert A \zeta)  +  h \Big( \frac{p(p-1)}{2 q'} + h^{\frac{p}{2}-1} \Big)   \vert  A \zeta \vert^p.
   \end{eqnarray*}
Taking the expectation    yields  (owing to the fact that $\mathbb E\,\zeta = 0$)
\[
\mathbb E  \vert a + \sqrt{h} A \zeta \vert^p   \le     \Big( 1 +  \frac{(p-1)(p-2)}{2 } h \Big)  \vert a \vert^p   +  h \Big(1+ p + h^{\frac{p}{2}-1} \Big)  \mathbb E \vert  A \zeta \vert^p.
\]
As a consequence, we get 
\[
\mathbb E  \vert a + \sqrt{h} A \zeta \vert^p   \le    \Big( 1 +  \frac{(p-1)(p-2)}{2 } h \Big)   \vert a \vert^p    +  h \Big( 1+p + h^{\frac{p}{2}-1} \Big)  \Vert A  \Vert^p \mathbb E \vert  \zeta \vert^p.
\]

\noindent $(b)$   
   It follows from the specified upper-bound of $a$  that  (keep in mind that $p \!\in (2,3]$)
  \setlength\arraycolsep{3pt}
   \begin{eqnarray*} 
   \vert  a  \vert^p  & \le & (1+2 L h )^p     \Big(   \frac{1+ L h }{1 + 2 L h } \vert x \vert   + \frac{L h }{1 + 2 L h } \Big)^p   \\
   & \le &   (1+2 L h )^p     \Big(   \frac{1+ L h }{1 + 2 L h } \vert x \vert^{p}   + \frac{L h }{1 + 2 L h } \Big) \\
   & \le & (1+2 L h )^p    \vert x \vert^{p}   +  (1+2 L h )^{p-1}  L h .
   \end{eqnarray*} 
 Then, combining this with  the specified upper-bound of $A$, we derive
  \begin{eqnarray*} 
   \mathbb E  \vert a + \sqrt{h} A \zeta \vert^p  & \le &    \Big( 1 +  \frac{(p-1)(p-2)}{2 } h \Big)    (1+ 2 L h )^p \vert x \vert^p  \\
   & & +   \Big( 1 +  \frac{(p-1)(p-2)}{2 } h \Big)    (1+ 2 L h )^{p-1} L  h  \\
   & &  +  h 2^{p-1}\Lambda^p  \Big( 1 + p + h^{\frac{p}{2}-1} \Big)  (1+\vert x \vert^p) \, \mathbb E \vert  \zeta \vert^p.
   \end{eqnarray*} 
Using the inequality $ 1+ u \le e^{u}$, for every $u\!\in \mathbb R$, we finally get 
 \begin{equation*} 
   \mathbb E  \vert a + \sqrt{h} A \zeta \vert^p   \le     \Big( e^{ \kappa_p h}  + K_p h  \Big)  \vert x \vert^p  +  \big(  e^{\kappa_p  h } L + K_p \big) h, 
   \end{equation*} 
where $\kappa_p := \Big(\frac{(p-1)(p-2)}{2 } + 2 p L \Big)$  and $K_p := 2^{p-1}\Lambda^p \Big( 1 + p + h^{\frac{p}{2}-1} \Big)   \mathbb E \vert  \zeta \vert^p$.
\end{proof}

\section{Gradient and Hessian of the recursive jump diffusion distortion when $\nu $ is Gaussian}  \label{AppenB}
 Using some standard computations similar to~\cite{PagSag} the components of the gradient vector of the distortion function read (in the short time framework)
for every  $j=1,\ldots,N_{k+1}$
\begin{eqnarray*}
  \frac{\partial \tilde D_{k+1}(\Gamma_{k+1}) }{\partial  x_{k+1}^j} & = &  2 \sum_{i=1}^{N_k} \sum_{m=0}^{1} p_k^i \, p_m \Big[ \big( x_{k+1}^j - \mu_k(m,x_{k}^i)\big)  \Big(  \Phi_0( x_{k+1}^{j+}(m,x_{k}^i))  -  \Phi_0(x_{k+1}^{j-}( m,x_{k}^i)) \Big) \\
  &  & \qquad  + \,  v_k(m, x_{k}^i) \big(\Phi_0'(  x_{k+1}^{j+}(m,  x_{k}^i)) - \Phi_0'( x_{k+1}^{j-}(m,x_{k}^i)) \big) \Big].
   \end{eqnarray*}
 The  diagonal terms of the Hessian matrix $\nabla^2  \widetilde D_{k+1} (\Gamma_{k+1}) $ are given by:
\begin{eqnarray*}
  \frac{\partial^2 \tilde D_{k+1}(\Gamma_{k+1}) }{\partial^2  x_{k+1}^j} = 2 \sum_{i=1}^{N_k} \sum_{m=0}^{1} &  p_k^i\, p_m  \Big[ & \Phi_0( x_{k+1}^{j+}(m, x_{k}^i))  -  \Phi_0( x_{k+1}^{j-}(m, x_{k}^i)) 
  \\
&  & - \frac{1}{4 v_k(m, x_{k}^i)} \Phi_0'( x_{k+1}^{j+}(m, x_{k}^i))( x_{k+1}^{j+1} - x_{k+1}^{j}) \\
&  &-  \frac{1}{4 v_k( m,x_{k}^i)} \Phi_0'(x_{k+1}^{j-}(m, x_{k}^i))( x_{k+1}^{j} - x_{k+1}^{j-1}) \Big] 
\end{eqnarray*}
and its  sub-diagonal terms are  
 \begin{equation*}
 \frac{\partial^2 \tilde D_{k+1}(\Gamma_{k+1}) }{\partial  x_{k+1}^j \partial  x_{k+1}^{j-1} } = -\frac{1}{2} \sum_{i=1}^{N_k} \sum_{m=0}^{1}   p_k^i \, p_m  \frac{1}{v_k( m,x_{k}^i)} ( x_{k+1}^j - x_{k+1}^{j-1}) \Phi_0'( x_{k+1}^{j-}(m, x_{k}^i)).
 \end{equation*}
 The upper-diagonals terms are
  \begin{equation*}
\frac{\partial^2 \tilde D_{k+1}(\Gamma_{k+1}) }{\partial  x_{k+1}^j \partial x_{k+1}^{j+1} } = -\frac{1}{2} \sum_{i=1}^{N_k} \sum_{m=0}^{1} p_k^i\, p_m  \frac{1}{v_k( m,x_{k}^i)} ( x_{k+1}^{j+1} -  x_{k+1}^{j}) \Phi_0'( x_{k+1}^{j+}(m, x_{k}^i)).
\end{equation*}

\section{Gradient and Hessian of the recursive jump diffusion distortion for a general  $\nu $}  \label{AppenC}
In this case, the components of the gradient vector  of the distortion function read
for every  $j=1,\ldots,N_{k+1}$
\begin{eqnarray*}
  \frac{\partial \tilde D_{k+1}(\Gamma_{k+1}) }{\partial  x_{k+1}^j} & = & 2  \sum_{i=1}^{N_k} \sum_{m=0}^{1} p_k^i \, p_m \int_{\mathbb R} \nu(du) \Big[ \big( x_{k+1}^j - \mu_k(m,x_{k}^i,u)\big)  \Big(  \Phi_0( x_{k+1}^{j+}(m,x_{k}^i,u))   \\
  &  & \   - \,  \Phi_0\big(x_{k+1}^{j-}( m,x_{k}^i,u)\big) \Big)  +   v_k( x_{k}^i) \big(\Phi_0'(  x_{k+1}^{j+}(m,  x_{k}^i,u)) - \Phi_0'( x_{k+1}^{j-}(m,x_{k}^i,u)) \big) \Big].
   \end{eqnarray*}
 The  diagonal terms of the Hessian matrix $\nabla^2  \tilde D_{k+1} (\Gamma_{k+1}) $ are given by:
\begin{eqnarray*}
  \frac{\partial^2 \tilde D_{k+1}(\Gamma_{k+1}) }{\partial^2  x_{k+1}^j} &=&  2 \sum_{i=1}^{N_k} \sum_{m=0}^{1}   p_k^i\, p_m \int_{\mathbb R} \nu(du) \Big[  \Phi_0( x_{k+1}^{j+}(m, x_{k}^i,u))  -  \Phi_0( x_{k+1}^{j-}(m, x_{k}^i,u)) 
  \\
&  &  \hspace{4.0cm}\,  - \, \frac{1}{4 v_k( x_{k}^i)} \Phi_0'( x_{k+1}^{j+}(m, x_{k}^i,u))( x_{k+1}^{j+1} - x_{k+1}^{j}) \\
&  & \hspace{4.0cm} \,- \, \frac{1}{4 v_k( x_{k}^i)} \Phi_0'(x_{k+1}^{j-}(m, x_{k}^i,u))( x_{k+1}^{j} - x_{k+1}^{j-1}) \Big] 
\end{eqnarray*}
and its  sub-diagonal terms are  
 \begin{equation*}
 \frac{\partial^2 \tilde D_{k+1}(\Gamma_{k+1}) }{\partial  x_{k+1}^j \partial  x_{k+1}^{j-1} } =  -\frac{1}{2} \sum_{i=1}^{N_k} \sum_{m=0}^{1}   p_k^i \, p_m \frac{1}{v_k( x_{k}^i)} ( x_{k+1}^j - x_{k+1}^{j-1})  \int_{\mathbb R}    \Phi_0'( x_{k+1}^{j-}(m, x_{k}^i,u)) \nu(du) .
 \end{equation*}
 The upper-diagonals terms are
  \begin{equation*}
\frac{\partial^2 \tilde D_{k+1}(\Gamma_{k+1}) }{\partial  x_{k+1}^j \partial x_{k+1}^{j+1} } = -\frac{1}{2} \sum_{i=1}^{N_k} \sum_{m=0}^{1} p_k^i\, p_m  \frac{1}{v_k( x_{k}^i)} ( x_{k+1}^{j+1} -  x_{k+1}^{j})  \int_{\mathbb R} \Phi_0'( x_{k+1}^{j+}(m, x_{k}^i,u)) \nu(du).
\end{equation*}
The involved functions are defined as follows: $\mu_k(m,x,u) =  x + h\,  b(t_k,x) + (m u - \lambda h \, \mathbb E(U_1) ) \gamma(x) $ and $v_k(x) = \sqrt{h} \, \sigma(t_k,x)$. Like previously, we set 
$$ 
x_{k+1}^{j-}(m,x,u): = \frac{x_{k+1}^{j-1/2} - \mu_{k}(m,x,u) }{v_k(x)} \  \textrm{ and } \  x_{k+1}^{j+}(m,x,u): = \frac{x_{k+1}^{j+1/2} - \mu_{k}(m,x,u) }{v_k(x)},\; k=0,\ldots,n-1.
$$

\end{appendix}

\end{document}